\documentclass[10pt,twoside]{scrarticle}
\usepackage[utf8]{inputenc}

\title{\uppercase{\textbf{\large{The cyclosyntomic regulator of a number field}}}}

\author{\textsc{tess bouis}\thanks{Institute for Advanced Study, Princeton, \texttt{tbouis@ias.edu}} \textsc{ and} $q$-\textsc{uentin gazda}\footnote{Sorbonne Université and Université Paris Cité, CNRS, IMJ-PRG, 75005 Paris, \texttt{quentin@gazda.fr}}}
\date{}
\usepackage[left=3cm, right=3cm, top=1in, bottom=1in, headheight=2cm]{geometry}

\usepackage{fancyhdr}
\usepackage{stmaryrd}

\usepackage{mathdots}

\usepackage{indentfirst}

\usepackage{xfrac}

\usepackage{amsmath, amscd, amsfonts, amssymb, amsthm,todonotes,eurosym,enumitem,relsize}

\usepackage[utf8]{inputenc}
\usepackage{dirtytalk}
\usepackage[T1]{fontenc}
\usepackage{mathrsfs}

\usepackage[colorlinks]{hyperref}
\usepackage[figure,table]{hypcap}
\definecolor{imperialred}{RGB}{237, 41, 57}
\definecolor{royalblue}{RGB}{64, 106, 212}
\definecolor{link}{RGB}{11,0,128}
\definecolor{gren}{RGB}{32,130,63}
\hypersetup{
	bookmarksnumbered,
	pdfstartview={FitH},
	citecolor=royalblue,
	linkcolor=imperialred,
	urlcolor=link,
	linktocpage = true
	pdfpagemode={UseOutlines}
}

\usepackage{xcolor}
\usepackage{comment}

\usepackage[object=pgfhan]{pgfornament}

\usepackage{graphicx}
\usepackage{array}
\usepackage{csquotes}
\usepackage{extarrows}

\usepackage{needspace}

\input{xy}
\xyoption{all}
\usepackage{tikz-cd}
\usepackage{textcomp}
\usepackage{colonequals}

\usepackage{sectsty}
\sectionfont{\textsc\centering}
\subsectionfont{\textbf\large}
\subsubsectionfont{\textbf\normalsize}

\usepackage{fancyhdr}

\newlength{\outermargin} 
\newlength{\mar} 
\newlength{\len}
\newlength{\temp}
\newtheorem{theorem}{Theorem}[section]
\newtheorem{lemma}[theorem]{Lemma}

\newtheorem{corollary}[theorem]{Corollary}
\newtheorem{proposition}[theorem]{Proposition}

\newtheorem{theoremintro}{Theorem}

\theoremstyle{definition}
\newtheorem{remark}[theorem]{Remark}

\newtheorem{examples}[theorem]{Examples}
\newtheorem{construction}[theorem]{Construction}
\newtheorem{definition}[theorem]{Definition}
\newtheorem{notation}[theorem]{Notation}

\newtheorem{variant}[theorem]{Variant}
\newtheorem{warning}[theorem]{Warning}

\DeclareMathOperator{\Z}{\mathbb{Z}}
\DeclareMathOperator{\Q}{\mathbb{Q}}

\DeclareMathOperator{\C}{\mathbb{C}}

\renewcommand{\epsilon}{\varepsilon}
\newcommand{\Li}{\operatorname{Li}}

\DeclareMathOperator{\G}{\mathbb{G}}

\newcommand{\nbrrg}{\mathcal{O}_K[\Delta_K^{-1}]}

\newcommand{\CR}{\mathcal{C}_{R}}
\newcommand{\can}{\text{can}}

\newcommand{\CRm}{\mathcal{C}_{R,m}}

\newcommand{\Frob}{\mathrm{Frob}}

\newcommand{\Cyc}{\mathrm{cyc}}
\newcommand{\Hab}{\mathrm{Hab}}

\newcommand{\CycSyn}{\mathrm{CycSyn}}

\newcommand{\W}{\mathbb{W}}
\newcommand{\qW}{q\text{-}\mathbb{W}}
\newcommand{\Wm}{\mathbb{W}_m}
\newcommand{\qWm}{q\text{-}\mathbb{W}_m}
\newcommand{\qWmm}{q\text{-}\mathbb{W}_{m'}}

\newcommand{\gh}{\mathrm{gh}}

\renewcommand{\ge}{\geqslant}
\renewcommand{\le}{\leqslant}
\renewcommand{\geq}{\geqslant}

\newcommand{\Hom}{\operatorname{Hom}}

\newcommand{\fib}{\operatorname{fib}}

\newcommand{\reg}{\mathrm{reg}}

\newcommand{\dlog}{\mathrm{dlog}}

\DeclareFontFamily{U}{MnSymbolC}{}
\DeclareFontShape{U}{MnSymbolC}{m}{n}{
	<-5.5> MnSymbolC5
	<5.5-6.5> MnSymbolC6
	<6.5-7.5> MnSymbolC7
	<7.5-8.5> MnSymbolC8
	<8.5-9.5> MnSymbolC9
	<9.5-11.5> MnSymbolC10
	<11.5-> MnSymbolCb12
}{}

\mathcode`A="7041 \mathcode`B="7042 \mathcode`C="7043 \mathcode`D="7044
\mathcode`E="7045 \mathcode`F="7046 \mathcode`G="7047 \mathcode`H="7048
\mathcode`I="7049 \mathcode`J="704A \mathcode`K="704B \mathcode`L="704C
\mathcode`M="704D \mathcode`N="704E \mathcode`O="704F \mathcode`P="7050
\mathcode`Q="7051 \mathcode`R="7052 \mathcode`S="7053 \mathcode`T="7054
\mathcode`U="7055 \mathcode`V="7056 \mathcode`W="7057 \mathcode`X="7058
\mathcode`Y="7059 \mathcode`Z="705A

\usepackage[bbgreekl]{mathbbol}
\DeclareSymbolFontAlphabet{\mathbb}{AMSb}
\DeclareSymbolFontAlphabet{\mathbbl}{bbold}
\newcommand{\Prism}{{\mathlarger{\mathbbl{\Delta}}}}

\usepackage{soul}
\usepackage{enumitem}
\numberwithin{equation}{theorem}

\usepackage{calrsfs}
\DeclareMathAlphabet{\pazocal}{OMS}{zplm}{m}{n}

\begin{document}
	
	\pagestyle{fancy}
	\fancyhead[EC]{TESS BOUIS AND QUENTIN GAZDA}
	\fancyhead[OC]{CYCLOSYNTOMIC REGULATOR}
	\fancyfoot[C]{\thepage}

    \maketitle

    \vspace{-3em}
	
	\begin{abstract}
        We construct a $q$-deformation of the $p$-adic regulator of a number field, called the \emph{cyclosyntomic regulator}, building on the Habiro ring of Garoufalidis--Scholze--Wheeler--Zagier. The key new ingredient in our construction is a refinement of Sulyma's norm maps in prismatic cohomology, which interpolate between classical powers and Frobenius maps at various prime numbers $p$. Furthermore, we compute the values of the cyclosyntomic regulator at units of the form $1-\zeta$, where $\zeta$ is a root of unity.
	\end{abstract}

	{
		\hypersetup{linkcolor=black}
		\tableofcontents
	}

    \section{Introduction}

    Motivated by questions on the emerging theory of $p$-adic $L$-functions, 
    Leopoldt \cite{leopoldt_arithmetik_1962} defined the $p$-adic regulator map of a number field $K$
    $$\reg_p : \mathcal{O}_K^\times \lhook\joinrel\longrightarrow  \prod_{\sigma \in \Hom(K,\C_p)} \C_p^{\times} \xlongrightarrow{\log_p} \prod_{\sigma \in \Hom(K,\C_p)} \C_p$$
    as a non-archimedean analogue of the classical regulator map. This $p$-adic regulator, via the $p$-adic class number formula proved by Colmez \cite{colmez_residu_1988}, is in particular related to the residue at $s=1$ of the $p$-adic Dedekind zeta function of $K$. The Leopoldt conjecture \cite{leopoldt_arithmetik_1962,leopoldt_p-adische_1975}, stating that the image of the $p$-adic regulator map is a lattice of rank $r_1 + r_2 - 1$, is however still open beyond the case where $K$ is an abelian extension of $\Q$ or of an imaginary quadratic number field (see \cite[Chapter~X.3]{meukirch_cohomology_2008} for a review). 

    The modern theory of regulators, as initiated by Beilinson \cite{beilinson_higher_1984}, recasts the regulator map of a number field as an instance of higher Chern class maps to Deligne cohomology. In the $p$\nobreakdash-adic setting, it was first noted by Kato \cite[Remark~$3.5$]{kato_p-adic_1987} that syntomic cohomology should be a good analogue of Deligne cohomology, and such a cohomological approach to $p$-adic regulators was successfully used by Gros and Kurihara \cite{gros_regulateurs_1990,gros_regulateurs_1994} to open the study of the $p$-adic Beilinson conjectures. 

    Recall that syntomic cohomology was first introduced by Fontaine--Messing \cite{fontaine_p-adic_1987} in terms of $p$-adic de Rham cohomology, and that a well-behaved integral refinement of their definition was introduced by Bhatt--Morrow--Scholze \cite{bhatt_topological_2019} and Bhatt--Scholze \cite{bhatt_prisms_2022} in terms of prismatic cohomology. Following the latter approach, the $p$-adic regulator map of a number field $K$ is then induced by the syntomic first Chern class
    $$c_1^{\text{syn}} : \G_m(R) \longrightarrow \text{H}^1_{\text{syn}}(R,\Z_p(1))$$
    of $R:=\mathcal{O}_K$, where the weight one syntomic complex of $R$ is defined, in terms of the Breuil--Kisin twisted $\Prism_R\{1\}$ and Nygaard filtered $\mathcal{N}^{\ge 1} \Prism_R\{1\}$ prismatic cohomology of $R$, by
    $$R\Gamma_{\text{syn}}(R,\Z_p(1)) := \fib\Big(\mathcal{N}^{\ge 1} \Prism_R \{1\} \xlongrightarrow{\can - \Frob_p^{\Prism}\{1\}} \Prism_R \{1\}\Big).$$

    The goal of this article is to introduce and study a decompleted refinement of this story (\emph{i.e.},~where the cohomology groups are not $p$-complete), which in particular allows us to interpolate the $p$-adic regulator maps between different prime numbers~$p$.

    \begin{definition}[Cyclosyntomic cohomology; see Definition~\ref{definitioncyclosyntomiccohomology}]\label{definitionintrocyclosyntomiccohomology}
        Let $K$ be a number field, $\Delta_K$ be the discriminant of $K$, $R$ be the étale $\Z$-algebra $\nbrrg$, and $d \ge 2$ be an integer. The \emph{cyclosyntomic complex of $R$} is the complex
        $$R\Gamma_{\CycSyn}(R,\Z(1)^{(d)}) := \left[ \mathcal{N}^{\ge 1} \mathcal{C}_R\{1\} \xlongrightarrow{\can - \Frob^{\Cyc}_d\{1\}} \mathcal{C}_R^{(d)}\{1\} \right]$$
        in the derived category $\mathcal{D}(\Z)$, where $\mathcal{N}^{\ge 1} \CR\{1\}$ and $\mathcal{C}_{R}^{(d)}\{1\}$ are $\mathbb{Z}[q]$-modules constructed as ``decompleted versions'' of the prismatic objects $\mathcal{N}^{\ge 1} \Prism_R \{1\}$ and $\Prism_R\{1\}$. 
    \end{definition}

    To define the objects appearing in Definition~\ref{definitionintrocyclosyntomiccohomology}, we use in a crucial way the recent definition of Habiro ring $\mathcal{H}_R$ of Garoufalidis--Scholze--Wheeler--Zagier \cite{garoufalidis_habiro_2024}, which can be seen as a ring of ``analytic functions in one variable $q$ at roots of unity''. For instance, $\mathcal{H}_{\Z} := \lim_{m \ge 1} \Z[q]^\wedge_{(q^m-1)}$ where the limit of the completions is taken over the set of integers $m \ge 1$ partially ordered by divisibility.  
    We refer to Section~\ref{subsectionHabirorings} for the definition of these Habiro rings, and to \cite{scholze_habiro_2025,wagner_q-hodge_2025,garoufalidis_explicit_2025} for recent progress towards a more general theory of Habiro cohomology. 
    
    More precisely, we introduce the notion of \emph{cyclotomic ring} $\mathcal{C}_R$ of a commutative ring $R$, as a reduced version of the Habiro ring $\mathcal{H}_R$ (Examples~\ref{examplescyclotomicring}\,(2)), in terms of the theory of big $q$\nobreakdash-Witt vectors developed by Wagner \cite{wagner_q-witt_2024}. For instance, $\mathcal{C}_{\Z} := \lim_{m \ge 1} \Z[q]/(q^m-1)$. 
    These cyclotomic rings $\mathcal{C}_R$ are commutative $\Z[q]$-algebras, and we define its cousins $\mathcal{N}^{\ge 1} \mathcal{C}_R \{1\}$ and $\mathcal{C}_R^{(d)}\{1\}$ as $\CR$\nobreakdash-modules which are compatible with the definitions of the Nygaard filtered and Breuil\nobreakdash--Kisin twisted prismatic complexes \cite{bhatt_topological_2019,bhatt_prisms_2022,bhatt_absolute_2022,antieau_prismatic_2023}. 

    \begin{theoremintro}[First Chern class; see Sections~\ref{subsectionthecyclosyntomiccohomologyofanumberfield} and~\ref{subsec:the-first-chern-class}] \label{theoremintrofirstChernclass}
        Let $K$ be a number field, and $R$ be the étale $\Z$-algebra $\nbrrg$. For every prime number $p$, there is a natural \emph{$p$-adic realisation map}
        $$R\Gamma_{\CycSyn}(R,\Z(1)^{(p)}) \longrightarrow R\Gamma_{q\emph{syn}}(R,\Z_p(1))$$
        to the syntomic cohomology of $R$ relative to the $q$-prism $(\Z_p[\![q-1]\!],[p]_q)$. Moreover, for every integer $d \ge 2$, there exists a natural \emph{cyclosyntomic first Chern class}
        $$c_1^{\CycSyn} : \G_m(R)[-1] \longrightarrow R\Gamma_{\CycSyn}(R,\Z(1)^{(d)})$$
        which recovers, for $d=p$ and after post-composing with the previous $p$-adic realisation map, the syntomic first Chern class of $R$.
    \end{theoremintro}

    By the work of Bhatt--Lurie \cite{bhatt_absolute_2022}, the theory of integral syntomic cohomology is equipped with a natural theory of Chern classes, which is built out of their notion of prismatic logarithm~$\log_{\Prism}$. A~derived refinement $\dlog_{\Prism}$ of their construction was more recently introduced by Mao \cite{mao_prismatic_2024}, and we similarly define the notion of \emph{cyclotomic logarithm}
    $$\dlog_{\Cyc} : \G_m(R)[-1] \longrightarrow \mathcal{N}^{\ge 1} \CR\{1\}$$
    to construct the cyclosyntomic first Chern class of Theorem~\ref{theoremintrofirstChernclass}. The key to define this cyclotomic logarithm is a refinement of Sulyma's norm maps in prismatic cohomology \cite{sulyma_prisms_2023,mao_prismatic_2024}, which we call the \emph{cyclotomic norms} (Proposition~\ref{prop:q-angeltveit}). 
    
    Note that, although one may define the objects $\mathcal{N}^{\ge 1} \CR\{1\}$ and $\CR\{1\}$ purely in terms of the $q$-Witt vectors of \cite{wagner_q-witt_2024}, we do need to use the Habiro rings of \cite{garoufalidis_habiro_2024,wagner_q-hodge_2025} to define the previous cyclotomic logarithm (Construction~\ref{constructiontruncatedcyclotomiclogarithm}). Relatedly, one can prove that this cyclotomic logarithm cannot be lifted to a nontrivial notion of Habiro logarithm (Remark~\ref{remarkHabirologarithmcannotexist}), which is our main reason for working at the level of cyclotomic rings rather than at the level of Habiro rings.

    Using these cohomological constructions, we now return to the number theoretic questions that motivated the study of these regulators. In the $p$-adic setting, an important application of syntomic cohomology is Kato's cohomological interpretation of explicit reciprocity laws \cite{kato_explicit_1991}. In this paper, Kato gives in particular the fundamental expression
    $$c_1^{\text{syn}} : u \longmapsto \frac{1}{p} \log\Big(\frac{u^p}{\Frob_p(u)}\Big) \in \text{H}^1_{\text{syn}}(R,\Z_p(1))$$
    for the syntomic first Chern class of a unit $u \in R^{\times}$ (\cite[Corollary~2.9]{kato_explicit_1991}, see also \cite[Proposition~4.1]{gros_regulateurs_1990}). The following result is the analogous computation in our context.

    \begin{theoremintro}[See Corollary~\ref{corollaryformulaforfirstChernclass}]\label{theoremintrocomputationfirstChernclass}
        Let $K$ be a number field, $R$ be the étale $\Z$-algebra $\nbrrg$, and $d \ge 2$ be an integer. For every unit $u \in R^\times$, the cyclosyntomic first Chern class at $u$ is given by
        $$c_1^{\CycSyn} : u \longmapsto \frac{1}{d} \log\Big(\frac{\tilde{\Pi}_d(\tilde{u})}{\Frob_d^{\Hab}(\tilde{u})}\Big) \in \emph{H}^1_{\CycSyn}(R,\Z(1)^{(d)})$$
        where $\tilde{\Pi}_d$ is a $q$-deformation of the $d^{th}$ power map (Proposition~\ref{propositionliftedcyclotomicpower}), $\Frob_d^{\Hab}$ is the $d^{th}$ Habiro Frobenius (Construction~\ref{constructionHabiroFrobenii}), and $\tilde{u} := (\Pi_m^{\Hab}(u))_{m \ge 1}$ is any sequence of Habiro lifts $\Pi_m^{\Hab}(u)$ in $\mathcal{H}_{R,m}/(q^m-1)^2$ of the cyclotomic norms $\Pi_m(u) \in \mathcal{H}_{R,m}/(q^m-1)$ of $u$ (Proposition~\ref{prop:q-angeltveit}).
    \end{theoremintro}

    Unwinding the definitions, Theorem~\ref{theoremintrocomputationfirstChernclass} boils down to constructing suitable homotopies of the desired form modulo $q^m-1$ which are compatible between different integers $m \ge 1$. We do so in Construction~\ref{constructionhomotopyforfirstChernclass}, after introducing the relevant Frobenius map on the Habiro ring of $R$ in Section~\ref{subsectionHabirorings}. 

    Although slightly intractable for completely general computations, the previous formula for the (cyclo)syntomic first Chern class can be made more explicit on the subgroup of cyclotomic units, \emph{i.e.}, on the subgroup generated by elements of the form $1-\zeta \in K^\times$ for $\zeta$ a root of unity. We refer to \cite[Chapter~8]{washington_1997} for a survey on cyclotomic units, and simply point out that cyclotomic units form a subgroup of finite index in the group of units of a cyclotomic field. 
    
    In the context of regulators, the role of these cyclotomic units was highlighted by Deligne's construction of a variation of mixed Hodge--Tate structure $\mathcal{L}\!\!\!\>\mathit{og}^H$ on $\mathbb{P}^1(\C)\setminus\{0,1,\infty\}$, called the \emph{\hbox{polylogarithm} variation}, which splits at roots of unity \cite{deligne_groupe_1989}. More concretely, this specific phenomenon at roots of unity implies that the special values of the first polylogarithm function
    $$\Li_1(z) := \sum_{k \ge 1} \frac{z^k}{k}$$
    at roots of unity can be seen as natural classes in weight one Deligne cohomology \cite{beilinson_interpretation_1994,huber_classical_1998}. A similar $p$-adic story of polylogarithms was initiated by Coleman \cite{coleman_dilogarithms_1982} (see also \cite[3.2]{deligne_groupe_1989}), who in particular defined the first $p$-adic polylogarithm function
    $$\Li_1^{(p)}(z) := \sum_{\substack{k \ge 1 \\ p \nmid k}} \frac{z^k}{k}$$
    whose special values at roots of unity were later proved to appear naturally as classes in weight one syntomic cohomology \cite{gros_regulateurs_1990}. 

    Given the prominent role of roots of unity in the context of Habiro cohomology, the following result can be seen as an attempt to better understand the relation between polylogarithms and cyclotomic units.
    
    \begin{theoremintro}[First $q$-polylogarithm as cyclosyntomic Chern class; see Theorem~\ref{theoremmainLi1cyclosyntomic}]\label{theoremintrofirstpolylogarithm}
        Let $K$ be a number field, $R$ be the étale $\Z$-algebra $\nbrrg$, and $d \ge 2$ be an integer. For every root of unity $\zeta \in R \setminus \{\pm 1\}$, the cyclosyntomic first Chern class at $1-\zeta \in R^\times$ is given by
        $$c_1^{\CycSyn} : 1 - \zeta \longmapsto - \Li_1^{(d)}([\zeta])_q \in \emph{H}^1_{\CycSyn}(R,\Z(1)^{(d)})$$
        where
        $$\Li_1^{(d)}(z)_q := \sum_{\substack{k \ge 1 \\ d \nmid k}} \frac{z^k}{[k]_q}$$
        is the first $q$-polylogarithm seen as a convergent $q$-power series in a suitable sense (Section~\ref{subsectionthefirstqpolylogarithm}) and $[\zeta]$ is a natural sequence of Habiro lifts of the cyclotomic norms of $\zeta$ (Notation~\ref{notationclassofzeta}). 
    \end{theoremintro}

    In particular, via the $p$-adic realisation map of Theorem~\ref{theoremintrofirstChernclass}, the cyclosyntomic cohomology classes of Theorem~\ref{theoremintrofirstpolylogarithm} also induce natural $q$-polylogarithm classes in the syntomic cohomology of $R$ relative to the $q$-prism. The proof of Theorem~\ref{theoremintrofirstpolylogarithm} essentially relies on the fact that cyclotomic units $u$ admit a natural explicit lift $\tilde{u}$ to the Habiro ring, which makes it possible to compute the Habiro Frobenius in Theorem~\ref{theoremintrocomputationfirstChernclass}.

    We end this introduction by mentioning two further potential directions of research. First, it is possible to assemble the cyclosyntomic complexes of Definition~\ref{definitionintrocyclosyntomiccohomology} for various $d \ge 2$ into a nontrivial complex of abelian groups. The resulting cyclosyntomic complex computes extensions of Tate objects in a category of $\CR$-modules equipped with compatible $d^{th}$ Frobenius maps for all integers $d \ge 2$. This phenomenon will be discussed by the second author in an epilogue to this article \cite{gazda-epilog}.
    
    Moreover, we expect that the first cyclosyntomic Chern class constructed here should be part of a more general theory of higher Chern classes, which would relate the higher $K$-groups of a number field, cyclosyntomic complexes of higher weights, and higher $q$-polylogarithms. While such a general theory seems currently out of reach, we expect that ongoing progress in the theory of Habiro cohomology may provide insight for further progress in this direction.

\subsection*{Notation.}

We denote by $\Delta_K$ the discriminant of a number field $K$. 
The $q$-analog $1+q+\cdots+q^{m-1} \in \Z[q]$ of an integer $m \ge 1$ is denoted by $[m]_q$.
For every integer $e \ge 1$, we denote by $\Phi_e(q) \in \Z[q]$ the $e^{th}$ cyclotomic polynomial. 
Given an integer $m \ge 1$ and a prime number $p$, we denote by $v_p(m)$ the $p$-adic valuation of $m$.
Given integers $a,b \ge 1$, we denote by $(a,b)$ (resp. $[a,b]$) the greatest common divisor (resp. the least common multiple) of $a$ and $b$.  
Given a commutative ring $R$, we denote by $\mathcal{D}(R)$ the derived $\infty$-category of $R$-modules.
We refer to \cite[Sections~I.1 and I.4]{atiyah_group_1969} for the definition and basic properties of $\Lambda$-rings. We call a $\Lambda$-ring $R$ perfect if its Adams operations $\psi^n : R \rightarrow R$ are isomorphisms. 

Finally, fix an integer $N \ge 1$ for the rest of this article, which represents a finite set of prime numbers that will be discarded in our constructions. Limits of the form $\lim_{m \ge 1}$ will typically be taken over the set of integers $m \ge 1$ that are coprime to $N$ (\emph{i.e.}, $(m,N)=1$), partially ordered by divisibility. The choice of the integer $N$ will be important only in Section~\ref{sec:first-polylog}, where we will need the fact that the root of unity $\zeta \in R$ is of order dividing $N$; 
before this point, the reader is welcome to assume that $N=1$.

\subsection*{Acknowledgements.}

We would like to thank Michel Gros, Peter Scholze, Ferdinand Wagner, and Campbell Wheeler for many inspiring and insightful discussions surrounding this project.
This project has received funding from the SFB 1085 Higher Invariants, the Institute for Advanced Study, and the Simons Foundation. This work also marks the launch of the ANR-25-CE40-3307-01 Al$K$traZ: \emph{Algebraic $K$\nobreakdash-theory, trace maps, and the Zagier conjecture}, of which both authors are members. We are pleased to express our gratitude to the ANR and to the anonymous experts for their valuable support of the project.

\section{Norm maps on $q$-Witt vectors}

In this section, we develop the necessary tools on big $q$-Witt vectors that we will use in the next sections. More precisely, we construct norm maps on the big $q$-Witt vectors of a commutative ring (Section~\ref{subsectionnormsonqWittvectors}), following the analogous theory of Angeltveit for classical big Witt vectors \cite{angeltveit_norm}, and introduce the notions of cyclotomic rings and of cyclotomic Frobenius (Section~\ref{subsectioncyclotomicrings}).

\subsection{Review of Witt vectors}\label{sectionreviewbigWittvectors}

In this subsection, we review the classical theory of big Witt vectors, and refer to the more extensive review \cite[Section~1]{hesselholt_deRham-Witt} for more details.

Let $R$ be a commutative ring. Given an integer $m \ge 1$, the \textit{$m$-truncated big Witt vectors $\Wm(R)$ of $R$} is the commutative ring characterised\footnote{The fact that this indeed characterises $\W_m(R)$ is a consequence of Lemma~\ref{lemmadwork}.} by the facts that it is abstractly isomorphic to the set $\prod_{e|m} R$, and that the map
$$\gh : \W_m(R) \longrightarrow \prod_{e|m}{R},\quad x := (x_e)_{e|m} \longmapsto (\mathrm{gh}_{e}(x))_{e|m}:=\left(\sum_{d|e}{dx_d^{e/d}}\right)_{e|m}$$
where the commutative ring structure on the target is defined componentwise, is a ring homomorphism. The map $\gh_e : \W_m(R) \rightarrow R$ is called the \emph{$e^{th}$ ghost coordinate}\footnote{This terminology is justified by the fact that the ghost map $\gh : \W_m(R) \rightarrow \prod_{e|m} R$ is injective for every flat $\Z$-algebra $R$ (or, slightly more generally, for every commutative ring $R$ which is $p$-torsionfree for every prime number $p$ dividing $m$).} (whereas the map $\W_m(R) \rightarrow R, \, (x_e)_{e|m} \mapsto x_e$ is called the \emph{$e^{th}$ Witt coordinate}).

Big Witt vectors are related to lifts of Frobenius via the following important lemma (\cite[Lemma~1.1]{hesselholt_deRham-Witt}). 

\begin{lemma}[Dwork's lemma]\label{lemmadwork}
    Let $R$ be a commutative ring, and $m \ge 1$ be an integer. Assume that for every prime number $p$ dividing $m$, there exists a ring homomorphism $\varphi_p : R \rightarrow R/p^{v_p(m)}R$ such that $\varphi_p(x) \equiv x^p$ modulo $pR$.\footnote{The statement \cite[Lemma~1.1]{hesselholt_deRham-Witt} is stated with the slightly stronger assumption that there exists a ring homomorphism $\varphi_p : R \rightarrow R$ lifting the Frobenius endomorphism of $R$. However, the proof of \cite[Lemma~1.1]{hesselholt_deRham-Witt} adapts readily to our context, which has the advantage to include all smooth $\Z$-algebras.} Then for every element $(x_e)_{e|m} \in \prod_{e|m} R$, the following are equivalent:
    \begin{enumerate}
        \item $(x_e)_{e|m} \in \prod_{e|m} R$ is in the image of the ghost map $\gh$;
        \item for every prime number $p$ dividing $m$ and every integer $e \ge 1$ satisfying $p | e |m$, we have $x_e \equiv \varphi_p(x_{e/p})$ modulo $p^{v_p(e)}R$.
    \end{enumerate}
\end{lemma}

\begin{definition}[Big Witt vectors]\label{definitionbigWittvectors}
    Let $R$ be a commutative ring. Given integers $m,m' \ge 1$ satisfying $m | m'$, there exists a natural \emph{restriction map} of commutative rings
    $$\text{R}_{m'/m} : \W_{m'}(R) \longrightarrow \W_m(R), \quad (x_e)_{e|m'} \longmapsto (x_e)_{e|m},$$
    and the \emph{big Witt vectors $\W(R)$ of $R$} is the commutative ring
    $$\W(R) := \lim_{m \ge 1} \W_m(R)$$
    where the transition maps are given by these restriction maps. Similarly, if one restricts to integers $m \ge 1$ which are the powers of a given prime number $p$, then this limit defines the ring of \emph{$p$-typical Witt vectors of $R$}.
\end{definition}

\begin{construction}[Frobenii and Verschiebungen]\label{constructionfrobeniiandverschiebungen}
    Let $R$ be a commutative ring. Given integers $m,m' \ge 1$ satisfying $m | m'$, and $d := \frac{m'}{m}$, there also exist natural \emph{Frobenius and Verschiebung maps}
    $$F_{d} : \W_{m'}(R) \longrightarrow \W_m(R) \quad \quad V_{d} : \W_m(R) \longrightarrow \W_{m'}(R)$$
    of commutative rings and of additive abelian groups, respectively. These maps are given, in ghost coordinates, by $F_{d} : (x_e)_{e|m'} \mapsto (x_{de})_{e|m}$ and $V_{d} : (x_e)_{e|m} \mapsto (dx_{\frac{e}{d}}\mathbbl{1}_{d|e})_{e|m'}$. They satisfy the relations $F_d \circ V_d = d$, 
    and $F_d \circ V_{d'} = V_{d'} \circ F_d$ for every integer $d' \ge 1$ coprime to $d$. 
\end{construction}

\begin{remark}
    The big Witt vectors $\W_m(R)$ of a commutative ring $R$ corresponds to the zeroth term of the big de Rham--Witt complex of $R$ (\cite{hesselholt_deRham-Witt}).
\end{remark}

\begin{remark}\label{remarkAinf}
    The limit $\lim_{m \ge 1} \W_m(R)$ along the Frobenius maps of Construction~\ref{constructionfrobeniiandverschiebungen} does not usually have a name. However, when $R:=\mathcal{O}_{\C_p}$ is the ring of integers of the $p$-adic complex numbers, and the limit is taken over the powers of the prime number $p$, this construction corresponds to Fontaine's period ring $A_{\text{inf}}$ in $p$-adic Hodge theory (\cite[Lemma~3.2]{bhatt_integral_2018}).
\end{remark}

\begin{notation}[Teichmüller lift]\label{notationteichmullerlift}
    Let $R$ be a commutative ring, and $m \ge 1$ be an integer. The \emph{Teichmüller lift on $R$} is the natural morphism of multiplicative monoids
    $$\Pi_m : R \longrightarrow \W_m(R)$$
    given by $x \mapsto (x^{m/e})_{e|m}$ in ghost coordinates.\footnote{The Teichmüller lift of $x$ is usually denoted by $[x]$. We choose the notation $\Pi_m(x)$ to avoid a potential confusion with the notation for $q$-integers $[k]_q \in \Z[q]$.} In particular, the Teichmüller lift is a section of the first ghost coordinate $\gh_1 : \W_m(R) \rightarrow R$.
\end{notation}

\subsection{Review of $q$-Witt vectors}

In this subsection, we review Wagner's theory of $q$-Witt vectors, as introduced in \cite[Section~2]{wagner_q-witt_2024}.

\begin{definition}[Big $q$-Witt vectors]\label{definitionqbigWittvectors}
    Let $R$ be a commutative ring. For every integer $m \ge 1$, the \emph{$m$-truncated big $q$-Witt vectors $\qWm(R)$ of $R$} is the commutative $\Z[q]$-algebra
    $$\qWm(R) := \W_m(R)[q]/\mathbb{I}_m$$
    where $\mathbb{I}_m \subseteq \W_m(R)[q]$ is the ideal generated by
    \begin{enumerate}[label=(\roman*)]
        \item $(q^e-1)\text{im}(V_{m/e})$, for every integer $e \ge 1$ satisfying $e|m$, and
        \item $\text{im}([e/d]_{q^d} V_{m/e} - V_{m/d} \circ F_{e/d})$, for all integers $d,e \ge 1$ satisfying $d|e|m$.
    \end{enumerate}
\end{definition}

\begin{examples}\label{examplesqWittvectors}
    Let $m \ge 1$ be an integer. The big $q$-Witt vectors of $R$ 
    can be made explicit in the following cases.
    \begin{enumerate}
        \item\label{item:Z} ($R=\Z$) $\qW_m(\mathbb{Z}) \cong \Z[q]/(q^m-1)$. More generally, for every perfect $\Lambda$-ring $R$ (\textit{e.g.}, $R=\Z_p$ for any prime number $p$, $R$ any commutative $\Q$-algebra, $R$ any perfect prism, or $R=\Z[T^{\Q}]$), $\qWm(R) \cong R[q]/(q^m-1)$ (\cite[Corollary~2.37]{wagner_q-witt_2024}).
        \item ($R=\nbrrg$) $\qWm(R) \cong \mathcal{H}_{R,m}/(q^m-1)$ for every number field $K$, where $R=\nbrrg$ and $\mathcal{H}_{R,m}$ is the $m$-truncated Habiro ring of $R$ (\cite[Theorem~2.9]{wagner_q-hodge_2025}).
        \item ($R=\Z[T]$) $\qW_m(\mathbb{Z}[T]) \cong \sum_{e|m}[e]_{q^{m/e}}\Z[T^{m/e},q]/(q^m-1) \subseteq \Z[T,q]/(q^m-1)$ (\cite[Proposition~2.36]{wagner_q-witt_2024}). Equivalently, $\qWm(R)$ is naturally identified with the kernel of the operator
        $$\mathbb{Z}[T,q]/(q^m-1)\longrightarrow \mathbb{Z}[T,q]/(q^m-1), \quad f(T,q) \longmapsto (q-1)\nabla_q(f(T,q))$$
        where $\nabla_q(f(T,q))$ is the $q$-derivative of $f(T,q)$ (see \cite[Theorem~4.27]{wagner_q-witt_2024} for a similar computation after $(q-1)$-completion).
    \end{enumerate}
\end{examples}

\begin{construction}[Frobenii and Verschiebungen]\label{constructionqfrobeniiandverschiebungen}
    Let $R$ be a commutative ring. Given integers $m, m' \ge 1$ satisfying $m|m'$, and $d := \frac{m'}{m}$, the big $q$-Witt vectors of $R$ are equipped with natural \emph{Frobenius and Verschiebung maps}
    $$F_d : \qWmm(R) \longrightarrow \qWm(R) \quad \quad V_d : \qWm(R) \longrightarrow \qWmm(R)$$
    of $\Z[q]$-algebras and of $\Z[q]$-modules, respectively (\cite[proof of Lemma~2.9]{wagner_q-witt_2024}). These are characterised by their $\Z[q]$-linearity and compatibility with the Frobenius and Verschiebung maps on big Witt vectors (Construction~\ref{constructionfrobeniiandverschiebungen}). They satisfy the relations $F_d \circ V_d = d$ and $V_d \circ F_d = [d]_{q^m}$.
\end{construction}

\begin{remark}\label{remarkuniversalpropertyqWittvectors}
    The big $q$-Witt vectors $\qWm(R)$ of a commutative ring $R$, together with their Frobenius and Verschiebung maps, satisfy a universal property: namely, the system of $\Z[q]$-algebras $(\qWm(R))_{m \ge 1}$ is the initial $q$-FV-system over $R$ (\cite[Definition~2.8 and Lemma~2.9]{wagner_q-witt_2024}). Also note that for every integer $m \ge 1$, we have $q^m-1 \in \mathbb{I}_m$ (Definition~\ref{definitionqbigWittvectors}\,(i) for $e=m$), so that $\qWm(R)$ is a commutative $\Z[q]/(q^m-1)$-algebra.
\end{remark}

\begin{warning}
    Despite the name, big $q$-Witt vectors are not a $q$-deformation of big Witt vectors, \textit{i.e.}, the natural morphism of commutative rings $\W_m(R) \rightarrow \qWm(R)/(q-1)$ is in general not an isomorphism (\cite[Remark~2.11]{wagner_q-witt_2024}).
\end{warning}

\begin{remark}\label{remarknorestriction} 
    The restriction maps $R_{m'/m} : \W_{m'}(R) \rightarrow \W_m(R)$ of Definition~\ref{definitionbigWittvectors} do not extend to morphisms of commutative $\Z[q]$-algebras $R_{m'/m} : \qWmm(R) \rightarrow \qWm(R)$ for general commutative rings $R$ (\cite[2.14]{wagner_q-witt_2024}). In particular, it is not possible to define the ring of big $q$-Witt vectors $\qW(R) := \lim_{m \ge 1} \qWm(R)$ as in Definition~\ref{definitionbigWittvectors}. Instead, in light of Remark~\ref{remarkAinf} and of the relation established in \cite[Section~2.2]{wagner_q-hodge_2025} between $q$-Witt vectors and Habiro rings, we will rather consider the limit $\lim_{m \ge 1} \qWm(R)$ along the Frobenius maps of Construction~\ref{constructionqfrobeniiandverschiebungen} in this paper, which we will typically refer to as \emph{cyclotomic rings} (Section~\ref{subsectioncyclotomicrings}).
\end{remark}

\subsection{Alternative approach to $q$-ghost maps}

In \cite[2.13]{wagner_q-witt_2024}, Wagner constructs $q$-ghost maps for the big $q$-Witt vectors as some explicit quotient by the images of some Verschiebung maps. In this subsection, we give an alternative description of these $q$-ghost maps (Definition~\ref{definitionqghostmaps}), using the universal property of big $q$-Witt vectors (Remark~\ref{remarkuniversalpropertyqWittvectors}). 

\begin{construction}[Cyclotomic $q$-FV-system]\label{constructioncyclotomicqFVsystem}
    Let $R$ be a commutative ring. Here we define a structure of $q$-FV-system over $R$ on the system of commutative $\Z[q]$-algebras $(\prod_{e|m} R[q]/\Phi_e(q))_{m \ge 1}$, in the sense of \cite[Definition~2.8]{wagner_q-witt_2024}. To do this, for every integer $m \ge 1$, let 
    $$\W_m(R)[q]/(q^m-1) \longrightarrow \prod_{e|m} R[q]/\Phi_e(q)$$
    be the morphism of commutative $\Z[q]$-algebras characterised by the fact that it is given by $$x \longmapsto (\gh_{m/e}(x))_{e|m}$$ on the subring $\W_m(R) \subseteq \W_m(R)[q]/(q^m-1)$.\footnote{Note that the order of the ghost coordinates is reversed compared to the convention used for ghost maps in Section~\ref{sectionreviewbigWittvectors}. This is related to the fact that there are no restriction maps in the theory of big $q$-Witt vectors (Remark~\ref{remarknorestriction}).} Given integers $m,m' \ge 1$ satisfying $m|m'$, and $d := \frac{m'}{m}$, let 
    $$F_{d} : \prod_{e|m'} R[q]/\Phi_e(q) \longrightarrow \prod_{e|m} R[q]/\Phi_e(q) \quad \quad V_{d} : \prod_{e|m} R[q]/\Phi_e(q) \longrightarrow \prod_{e|m'} R[q]/\Phi_e(q)$$
    be the morphisms given by $F_d : (c_e(q))_{e|m'} \mapsto (c_e(q))_{e|m}$ and $V_d : (c_e(q))_{e|m} \mapsto (dc_e(q)\mathbbl{1}_{e|m})_{e|m'}$, respectively. These morphisms satisfy the relations $F_d \circ V_d = d$ and $V_d \circ F_d = [d]_{q^m}$, and are compatible with the Frobenius and Verschiebung maps on big Witt vectors, thus endowing the system $(\prod_{e|m} R[q]/\Phi_e(q))_{m \ge 1}$ with the structure of $q$-FV-system over $R$.
\end{construction}

\begin{definition}[$q$-ghost maps]\label{definitionqghostmaps}
    Let $R$ be a commutative ring. For every integer $m \ge 1$, the \emph{$q$-ghost map} on big $q$-Witt vectors is the morphism of commutative $\Z[q]$-algebras
    $$q\text{-}\gh : \qWm(R) \longrightarrow \prod_{e|m} R[q]/\Phi_e(q), \quad c \longmapsto (c_e(q))_{e|m} := (q\text{-}\gh_e(c))$$
    induced by the universal property of $(\qWm(R))_{m \ge 1}$ as the initial $q$-FV-system over $R$ (Remark~\ref{remarkuniversalpropertyqWittvectors} and Construction~\ref{constructioncyclotomicqFVsystem}). The map $q\text{-}\gh_e : \qWm(R) \rightarrow R[q]/\Phi_e(q)$ is called the \emph{$e^{th}$ $q$-ghost coordinate}.\footnote{This terminology is justified by the result \cite[Lemma~2.23]{wagner_q-witt_2024}, stating that the $q$-ghost map \hbox{$q$-$\gh : \qWm(R) \rightarrow \prod_{e|m} R[q]/\Phi_e(q)$} is injective for every flat $\Z$-algebra~$R$ (or, slightly more generally, for every commutative ring~$R$ which is $p$-torsionfree for every prime number $p$ dividing $m$).}
\end{definition}

\begin{remark}
    The $q$-ghost maps of Definition~\ref{definitionqghostmaps} agree with the $q$-ghost maps defined in \cite[2.13]{wagner_q-witt_2024}. Indeed, one can prove that the $q$-ghost maps of \cite[2.13]{wagner_q-witt_2024} define a map of $q$-FV-systems over $R$ from $(\qWm(R))_{m \ge 1}$ to $\prod_{e|m} R[q]/\Phi_e(q)$ (equipped with the $q$-FV-system structure from Construction~\ref{constructioncyclotomicqFVsystem}), and such a map must be unique by universality (Remark~\ref{remarkuniversalpropertyqWittvectors}). 
\end{remark}

The following result is an analogue of Lemma~\ref{lemmadwork} for big $q$-Witt vectors.

\begin{lemma}[$q$-Dwork's lemma]\label{lemmaqDwork}
    Let $R$ be a commutative ring, and $m \ge 1$ be an integer. Assume that for every prime number $p$ dividing $m$, there exists a ring homomorphism $\varphi_p : R \rightarrow R/p^{v_p(m)}R$ such that $\varphi_p(x) \equiv x^p$ modulo $pR$. Then for every element $(c_e(q))_{e|m} \in \prod_{e|m} R[q]/\Phi_e(q)$, the following are equivalent:
    \begin{enumerate}
        \item $(c_e(q))_{e|m} \in \prod_{e|m} R[q]/\Phi_e(q)$ is in the image of the $q$-ghost map $q$-$\gh$;
        \item there exists a lift $(\tilde{c}_e(q))_{e|m} \in \prod_{e|m} R[q]$ of $(c_e(q))_{e|m} \in \prod_{e|m} R[q]/\Phi_e(q)$ such that, for every prime number $p$ and every integer $e \ge 1$ satisfying $pe|m$, we have $\tilde{c}_e(q) \equiv \varphi_p(\tilde{c}_{pe}(q))$ in $(R/p^{v_p(m/e)})[q]$, where $\varphi_p : R[q] \rightarrow (R/p^{v_p(m/e)})[q]$ is the unique $\Z[q]$-linear ring homomorphism extending $\varphi_p : R \rightarrow R/p^{v_p(m/e)}R$.
    \end{enumerate}
\end{lemma}

\begin{proof}
    By construction, there is a commutative diagram of $\Z[q]$-algebras
    $$\begin{tikzcd}
        \W_m(R)[q] \ar[r,"\text{gh}"] \ar[d,two heads] & \prod_{e|m} R[q] \ar[r,two heads] & \prod_{e|m} R[q] / \Phi_e(q) \ar[d] \\
        \qWm(R) \arrow[rr,"q\text{-gh}"] & & \prod_{e|m} R[q]/\Phi_e(q)
    \end{tikzcd}$$
    where the top horizontal map is induced by the ghost map of Section~\ref{sectionreviewbigWittvectors} and the right vertical map is given by $(c_e(q))_{e|m} \mapsto (c_{m/e}(q))_{e|m}$. The desired result is then a consequence of Dwork's lemma (Lemma~\ref{lemmadwork}).
\end{proof}

\begin{remark}\label{remarknaiveqDwork}
    Let $R$ be a flat $\Z$-algebra. 
    Lemma~\ref{lemmaqDwork} suggests that big $q$-Witt vectors $\qWm(R)$ are a sort of ``Frobenius twisted'' version of $R[q]/(q^m-1)$.\footnote{This is directly related to the definition of the Habiro rings of a number field in \cite{garoufalidis_habiro_2024,wagner_q-hodge_2025}, where the Habiro ring $\mathcal{H}_{R,m}$ of an étale $\Z$-algebra $R$ is a sort of ``Frobenius twisted'' version of $R[q]^\wedge_{(q^m-1)}$.} More precisely, one can prove that the image of the natural projection map $R[q]/(q^m-1) \rightarrow \prod_{e|m} R[q]/\Phi_e(q)$ is given by the elements $(c_e(q))_{e|m} \in \prod_{e|m} R[q]/\Phi_e(q)$ for which there exists a lift $(\tilde{c}_e(q))_{e|m} \in \prod_{e|m} R[q]$ such that, for every prime number $p$ and every integer $e \ge 1$ satisfying $pe|m$, we have $\tilde{c}_e(q) \equiv \tilde{c}_{pe}(q)$ in $(R/p^{v_p(m/e)})[q]$. 
\end{remark}

\begin{notation}[$q$-Teichmüller lift]\label{notationqteichmullerlift}
    Let $R$ be a commutative ring, and $m \ge 1$ be an integer. The \emph{$q$-Teichmüller lift on $R$} is the morphism of multiplicative monoids
    $$\Pi_m : R \longrightarrow \qWm(R)$$
    defined as the composition of the Teichmüller lift $\Pi_m : R \rightarrow \Wm(R)$ of Notation~\ref{notationteichmullerlift} with the canonical map \hbox{$\W_m(R) \rightarrow \qWm(R)$}. In particular, the $q$-Teichmüller lift is a section of the first $q$-ghost coordinate $q\text{-}\gh_1 : \qWm(R) \rightarrow R[q]/\Phi_1(q) \cong R$, and is in general given by $x \mapsto (x^{m/e})_{e|m}$ in $q$\nobreakdash-ghost coordinates.
\end{notation}

\subsection{Norm maps on $q$-Witt vectors}\label{subsectionnormsonqWittvectors}

In \cite{angeltveit_norm}, Angeltveit constructs multiplicative norm maps on big Witt vectors, as a natural generalisation of the Teichmüller lift. In this subsection, we prove that big $q$-Witt vectors similarly admit multiplicative norm maps $\Pi_{m'/m} : \qWm(R) \rightarrow \qWmm(R)$ (Proposition~\ref{prop:q-angeltveit}). 

\begin{lemma}\label{lemmanormpreserveimageofqghostmap}
    Let $R$ be a commutative ring, and $m,m' \ge 1$ be integers satisfying $m | m'$. Assume that for every prime number $p$ dividing $m$, there exists a ring homomorphism $\varphi_p : R \rightarrow R/p^{v_p(m')}R$ such that $\varphi_p \equiv x^p$ modulo $pR$. Then for every element $c \in \qWm(R)$ with $q$-ghost coordinates $(c_e(q))_{e|m}$,  
    the element $$\Pi_{m'/m}((c_e(q))_{e|m}) := \big(c_{(e,m)}(q^{\frac{e}{(e,m)}})^{\frac{m'}{[e,m]}}\big)_{e|m'} \in \prod_{e|m'} R[q]/\Phi_e(q)$$ is in the image of the $q$-ghost map $q\text{-}\gh : \qWmm(R) \rightarrow \prod_{e | m'} R[q]/\Phi_e(q)$.
\end{lemma}

\begin{proof}
    First note that the morphism of multiplicative monoids $$\Pi_{m'/m} : \prod_{e|m} R[q]/\Phi_e(q) \rightarrow \prod_{e|m'} R[q]/\Phi_e(q)$$ is well-defined, because $\Phi_e(q)$ divides $\Phi_{(e,m)}(q^{\frac{e}{(e,m)}})$ in the commutative ring $\Z[q]$. Moreover, for any integers $m,m',m'' \ge 1$ satisfying $m|m'|m''$, we have the equality $$\Pi_{m''/m'} \circ \Pi_{m'/m} = \Pi_{m''/m}.$$ Indeed, this is a consequence of the series of equalities
    $$c_{((e,m'),m)}\big(q^{\frac{(e,m')}{((e,m'),m)}\cdot\frac{e}{(e,m')}}\big)^{\frac{m'}{[(e,m'),m]}\cdot\frac{m''}{[e,m']}} = c_{(e,m)}(q^{\frac{e}{(e,m)}})^{\frac{m'}{[(e,m'),m]}\cdot\frac{m''}{[e,m']}} = c_{(e,m)}(q^{\frac{e}{(e,m)}})^{\frac{m''}{[e,m]}}$$
    where the equality $\tfrac{m'}{[(e,m'),m]\cdot[e,m']}=\tfrac{1}{[e,m]}$ follows from the fact that for any integers $a,b,c \ge 0$ satisfying $c \le b$, we have the equality 
    $b-\max(\min(a,b),c)-\max(a,b)=-\max(a,c)$. 
    
    In particular, it suffices to prove the desired claim when $m'/m$ is a prime number, which we assume from now on.
    
    Let $(c_e(q))_{e|m} \in \prod_{e|m} R[q]/\Phi_e(q)$ be an element in the image of the $q$-ghost map. We use the $q$-Dwork lemma (Lemma~\ref{lemmaqDwork}) to prove that $\Pi_{m'/m}((c_e(q))_{e|m}) \in \prod_{e|m'} R[q]/\Phi_e(q)$ is in the image of the $q$-ghost map. Let $(\tilde{c}_e(q))_{e|m} \in \prod_{e|m} R[q]$ be a lift of $(c_e(q))_{e|m} \in \prod_{e|m} R[q]/\Phi_e(q)$ such that, for every prime number $p$ and every integer $e \ge 1$ satisfying $pe|m$, we have $\tilde{c}_e(q) \equiv \varphi_p(\tilde{c}_{pe}(q))$ in $(R/p^{v_p(m/e)})[q]$. It then suffices to prove that, for every prime number $p$ and every integer $e \ge 1$ satisfying $pe|m'$, we have $\varphi_p\big(\tilde{c}_{(pe,m)}(q^{\frac{pe}{(pe,m)}})^{\frac{m'}{[pe,m]}}\big) \equiv \tilde{c}_{(e,m)}(q^{\frac{e}{(e,m)}})^{\frac{m'}{[e,m]}}$ in $(R/p^{v_p(m'/e)})[q]$. We distinguish several cases.
    
    Assume first that $v_p(e)>v_p(m)$. Because $m'/m$ is a prime number, this implies that $m'/m=p$, so the condition $pe|m'$ can not be satisfied in this case,  
    and the desired statement is vacuously true.

    Assume now that $v_p(e)=v_p(m)$. If $v_p(e)=v_p(m')$, then $v_p(m'/e)=0$ and the desired statement is again vacuously true, so we assume that $v_p(e)=v_p(m)<v_p(m')$. Because $m'/m$ is a prime number, this implies that $m'=pm$ and that $v_p(m'/e)=1$. The desired statement then follows from the series of equalities
    $$\varphi_p\big(\tilde{c}_{(pe,m)}(q^{\frac{pe}{(pe,m)}})^{\frac{m'}{[pe,m]}}\big) = \varphi_p\big(\tilde{c}_{(e,m)}(q^{\frac{pe}{(e,m)}})\big)^{\frac{m'}{p\cdot[e,m]}} \equiv \tilde{c}_{(e,m)}(q^{\frac{e}{(e,m)}})^{p\cdot\frac{m'}{p\cdot [e,m]}} = \tilde{c}_{(e,m)}(q^{\frac{e}{(e,m)}})^{\frac{m'}{[e,m]}}$$
    where the congruence holds modulo $p$, using that $\varphi_p : R \rightarrow R/p^{v_p(m')}R$ is a lift of Frobenius. 

    Assume finally that $v_p(e)<v_p(m)$. In this case, the desired statement follows similarly from the series of equalities
    $$\varphi_p\big(\tilde{c}_{(pe,m)}(q^{\frac{pe}{(pe,m)}})^{\frac{m'}{[pe,m]}}\big) = \varphi_p\big(\tilde{c}_{p\cdot(e,m)}(q^{\frac{pe}{p\cdot(e,m)}})\big)^{\frac{m'}{[e,m]}} \equiv \tilde{c}_{(e,m)}(q^{\frac{e}{(e,m)}})^{\frac{m'}{[e,m]}} = \tilde{c}_{(e,m)}(q^{\frac{e}{(e,m)}})^{\frac{m'}{[e,m]}}$$
    where the congruence holds modulo $p^{v_p(m'/e)}$, as a consequence of the assumption on $(\tilde{c}_e(q))_{e|m}$.  
    More precisely, either $m'/m$ is different from $p$, in which case $v_p(m'/e)=v_p(m/(e,m))$ and this is indeed a consequence of the assumption, or $m'/m=p$, in which case we use the fact that $a \equiv b$ modulo $p^{v_p(m/e)}$ implies $a^p \equiv b^p$ modulo $p^{v_p(m'/e)}$. 
\end{proof}

\begin{proposition}[Cyclotomic norms]\label{prop:q-angeltveit}
    Let $R$ be a commutative ring. Then for any integers $m,m' \ge 1$ satisfying $m | m'$, there exists a natural morphism of multiplicative monoids
    $$\Pi_{m'/m} : \qWm(R) \longrightarrow \qWmm(R)$$
    given by $(c_e(q))_{e|m} \mapsto (c_e'(q))_{e|m'} := \big(c_{(e,m)}(q^{\frac{e}{(e,m)}})^{\frac{m'}{[e,m]}}\big)_{e|m'}$ in $q$-ghost coordinates. This morphism is uniquely determined by the fact that it is functorial in $R$ and given by the previous formula on $q$-ghost coordinates.
\end{proposition}

\begin{proof}
    First assume that $R=\Z[T_i,i \in I]$ is a polynomial $\Z$-algebra (potentially in infinitely many variables). For every prime number $p$, let $\varphi_p : R \rightarrow R$ be the lift of Frobenius given by $T_i \mapsto T_i^p$. The $\Z$-algebra $R$ is flat, so the $q$-ghost maps of Definition~\ref{definitionqghostmaps} are injective (\cite[Lemma~2.23]{wagner_q-witt_2024}). The formula in $q$-ghost coordinates is moreover multiplicative by construction, so it suffices to prove that an element $(c_e(q))_{e|m}$ in the image of the $q$-ghost map is sent to an element $(c_e'(q))_{e|m'}$ in the image of the $q$-ghost map. Using the lifts of Frobenius $\varphi_p$, this is a consequence of Lemma~\ref{lemmanormpreserveimageofqghostmap}. 

    Assume now that $R$ is a general commutative ring. Let $P \twoheadrightarrow R$ be a surjective morphism of commutative rings, where $P$ is a polynomial $\Z$-algebra. There is a natural diagram of commutative $\Z[q]$-algebras
    $$\begin{tikzcd}[column sep = large]
        \qWm(P) \ar[r,"\Pi_{m'/m}"] \ar[d,two heads] & \qWmm(P) \ar[d,two heads] \\
        \qWm(R) \ar[r,"\Pi_{m'/m}",dotted] & \qWmm(R)
    \end{tikzcd}$$
    where the vertical maps are surjective (\cite[Corollary~2.29]{wagner_q-witt_2024}). In particular, given an element $c \in \qWm(R)$, any lift $\tilde{c} \in \qWm(P)$ induces an element $\Pi_{m'/m}(c) \in \qWmm(R)$, which is compatible with the desired formula on $q$-ghost coordinates by the previous paragraph and the functoriality of $q$-ghost maps. Note that this construction a priori depends on the surjection $P \twoheadrightarrow R$ and on the lift $\tilde{c} \in \qWm(R)$. To prove that it is independent of these choices, let $P_1 \twoheadrightarrow R$, $\tilde{c}_1 \in \qWm(P_1)$ and $P_2 \twoheadrightarrow R$, $\tilde{c}_2 \in \qWm(P_2)$ be two such choices. Let $P := P_1 \times P_2$. By \cite[Corollary~2.29]{wagner_q-witt_2024}, there is a natural isomorphism of commutative $\Z[q]$-algebras $\qWm(P) \cong \qWm(P_1) \times \qWm(P_2)$. In particular, there exists an element $\tilde{c} \in \qWm(R)$ mapping to $\tilde{c}_1 \in \qWm(P_1)$ and $\tilde{c}_2 \in \qWm(P_2)$ via the natural projection maps, and the element $\Pi_{m'/m}(c) \in \qWmm(R)$ is well-defined.
\end{proof}

\begin{remark}\label{remarkcompatibilityupgradedcyclotomicpowers}
    Let $R$ be a commutative ring. For any integers $m,m',m'' \ge 1$ satisfying $m|m'|m''$, the cyclotomic norms of Proposition~\ref{prop:q-angeltveit} satisfy the identity $\Pi_{m''/m'} \circ \Pi_{m'/m} \circ \Pi_{m''/m}$. This is indeed a consequence of the first paragraph of the proof of Lemma~\ref{lemmanormpreserveimageofqghostmap}. In particular, for any integers $m,m' \ge 1$ satisfying $m|m'$, we have the equality $\Pi_{m'/m} \circ \Pi_m = \Pi_{m'}$, where the notation here is compatible with that of the Teichmüller lift (Notation~\ref{notationqteichmullerlift}).
\end{remark}

\subsection{Cyclotomic rings}\label{subsectioncyclotomicrings}

In this subsection, we introduce the notion of cyclotomic ring $\mathcal{C}_R$ of a commutative ring $R$ (Definition~\ref{definitioncyclotomicring}).

\begin{definition}[Cyclotomic ring]\label{definitioncyclotomicring}
    Let $R$ be a commutative ring. The \emph{cyclotomic ring of $R$} is the commutative $\Z[q]$-algebra
    $$\CR := \lim_{m \ge 1} \qWm(R)$$
    where the limit is taken along the Frobenius maps of Construction~\ref{constructionqfrobeniiandverschiebungen}.
\end{definition}

\begin{examples}\label{examplescyclotomicring}
    Following Examples~\ref{examplesqWittvectors}, the cyclotomic ring of $R$ can be made explicit in the following cases.
    \begin{enumerate}
        \item ($R=\Z$) $\mathcal{C}_{\Z} \cong \lim_{m \ge 1} \Z[q]/(q^m-1)$ where the limit is taken along the natural projection maps, and similarly for any perfect $\Lambda$-ring $R$.
        \item ($R=\nbrrg$) $\CR \cong \lim_{m \ge 1} \mathcal{H}_{R,m}/(q^m-1)$ for every number field $K$, where $R=\nbrrg$ and the limit is taken along the natural projection maps (\cite[Remark~2.9]{wagner_q-hodge_2025}).
        \item ($R=\Z[T]$) $\mathcal{C}_{\Z[T]} \cong \mathcal{C}_{\Z}$ (Examples~\ref{examplesqWittvectors}\,(3)). 
    \end{enumerate}
\end{examples}

\begin{remark}
    For general commutative rings $R$, it may happen that the derived inverse limit $R\!\lim_{m \ge 1} \qWm(R)$, where the limit is taken along the Frobenius maps of Construction~\ref{constructionqfrobeniiandverschiebungen}, is not concentrated in degree zero. However, if $R$ is étale over $\Z$, or over a perfect $\Lambda$-ring, then one can use the $q$-Dwork lemma (Lemma~\ref{lemmaqDwork}) to prove that the Frobenius maps of Construction~\ref{constructionqfrobeniiandverschiebungen} are surjective, hence the derived and classical limits coincide.
\end{remark}

\begin{lemma}\label{lemmapq-1completion}
    For every prime number $p$ and every integer $r \ge 0$, the element $q^{p^r}-1$ of $\Z[q]$ belongs to the ideal $(p,q-1)^r \subseteq \Z[q]$. In particular, for every $p$-torsionfree commutative ring $R$, the $(p,q-1)$-completion of the commutative $\Z[q]$-algebra $\lim_{r \ge 0} R[q]/(q^{p^r}-1)$ is naturally identified with $R^\wedge_p[\![q-1]\!]$.
\end{lemma}

\begin{proof}
    The second claim is a direct consequence of the first claim. To prove the first claim, it suffices to prove that $[p^r]_q \in (p,q-1)^r$, because $q^{p^r}-1=(q-1) \cdot [p^r]_q$. This statement for $[p^r]_q$ is a consequence of the formula $[p^r]_q = \prod_{i=0}^{r-1} [p]_{q^{p^i}}$ and of the congruence $[p]_{q^{p^i}} \equiv p$ modulo $(q-1)$ for every integer $i \ge 0$.
\end{proof}

\begin{remark}[$p$-adic realisation]\label{remarkpadicrealisationcyclotomic}
    Let $R$ be a flat $\Z$-algebra, and $p$ be a prime number. If $R^\wedge_p$ admits a lift of Frobenius $\varphi_p$ (\textit{i.e.}, a ring homomorphism $\varphi_p : R^\wedge_p \rightarrow R^\wedge_p$ whose reduction modulo~$p$ is the Frobenius endomorphism of $R/p$), then there is a natural morphism of commutative $\Z[q]$\nobreakdash-algebras
    $$\mathcal{C}_R \longrightarrow R^\wedge_p[\![q-1]\!].$$
    Using Lemma~\ref{lemmaqDwork} and Remark~\ref{remarknaiveqDwork}, this morphism is induced by the composite morphism of commutative $\Z[q]$-algebras
    $$\prod_{e \ge 1} R[q]/\Phi_e(q) \longrightarrow \prod_{r \ge 0} R[q]/\Phi_{p^r}(q) \longrightarrow \prod_{r \ge 0} R[q]/\Phi_{p^r}(q)$$
    where the first map is the natural projection $(c_e(q))_{e \ge 1} \mapsto (c_{p^r}(q))_{r \ge 0}$ and the second map is given by $(c_{p^r}(q))_{r \ge 0} \mapsto (\varphi_p^r(c_{p^r}(q)))$.\footnote{Following Lemma~\ref{lemmaqDwork}, we denote by $\varphi_p : R[q]/\Phi_{p^r}(q) \rightarrow R[q]/\Phi_{p^r}(q)$ the unique $\Z[q]$-linear ring homomorphism extending $\varphi_p : R^\wedge_p \rightarrow R^\wedge_p$.} More precisely, the image of the cyclotomic ring $\mathcal{C}_{R}$ via this composite morphism is given by $\lim_{r \ge 0} R[q]/(q^{p^r}-1)$, and the desired morphism is obtained by post-composing with the $(p,q-1)$-completion $\lim_{r \ge 0} R[q]/(q^{p^r}-1) \rightarrow R^\wedge_p[\![q-1]\!]$ (Lemma~\ref{lemmapq-1completion}). 
    Note that this map $\mathcal{C}_R \rightarrow R^\wedge_p[\![q-1]\!]$ depends on the choice of Frobenius lift $\varphi_p$ (\textit{e.g.}, if $R$ is smooth over $\Z$). Also note that if $R$ is étale over $\Z$, then such a Frobenius lift $\varphi_p$ on $R^\wedge_p$ is unique.
\end{remark}

We now introduce the notion of cyclotomic Frobenius on these cyclotomic rings. To do so, we will use the following variant of Definition~\ref{definitionqbigWittvectors}.

\begin{definition}[Big $q$-Witt vectors at $d$]\label{definitionbigqWittvectorsatd}
    Let $R$ be a commutative ring, and $d \ge 1$ be an integer. For every integer $m \ge 1$, the \emph{$m$-truncated big $q$-Witt vectors at $d$ of $R$} is the commutative $\Z[q]$-algebra
    $$\qWm^{(d)}(R) := \Wm(R)[q]/\mathbb{I}_m^{(d)}$$
    where $\mathbb{I}^{(d)}_m \subseteq \W_m(R)[q]$ is the ideal generated by
    \begin{enumerate}[label=(\roman*)]
        \item $\big(\prod_{d|e'|e} \Phi_{e'}(q)\big)\text{im}(V_{m/e})$, for every integer $e \ge 1$ satisfying $e|m$,\footnote{Note here that $\prod_{d|e'|e} \Phi_{e'}(q) = 1$ if $d \nmid e$.} and
        \item $\text{im}([f/e]_{q^e} V_{m/f} - V_{m/e} \circ F_{f/e})$, for all integers $e,f \ge 1$ satisfying $d|e|f|m$.
    \end{enumerate}
\end{definition}

\begin{remark}
    For every commutative ring $R$ and integer $m \ge 1$, we have $\qWm^{(1)}(R) = \qWm(R)$. Moreover, for every integer $d \ge 1$, one can check that $\mathbb{I}_m \subseteq \mathbb{I}_m^{(d)}$.
\end{remark}

\begin{remark}\label{remarkuniversalpropertyqWittvectorsatd}
    Let $R$ be a commutative ring, and $d \ge 1$ be an integer. Arguing as in \cite[Lemma~2.9]{wagner_q-witt_2024}, the system $(\qWm^{(d)}(R))_{m \ge 1}$ is the initial \emph{$q$-FV-system at $d$ over $R$}, \emph{i.e.}, the initial system of commutative $\Z[q]$-algebras $(W_m^{(d)})_{m \ge 1}$ indexed by integers $m \ge 1$, together with the following structure:
    \begin{enumerate}[label=(\roman*)]
        \item for every integer $m \ge 1$, a morphism of $\Z[q]$-algebras $\W_m(R)[q]/\prod_{d|e|m}\Phi_e(q) \rightarrow W_m^{(d)}$ (in particular $W^{(d)}_m=0$ unless $d|m$);
        \item for any integers $m,m' \ge 1$ satisfying $d|m|m'$, morphisms $F_{m'/m} :W_{m'}^{(d)} \rightarrow W_m^{(d)}$ of $\Z[q]$\nobreakdash-algebras and $V_{m'/m} : W_m^{(d)} \rightarrow W_{m'}^{(d)}$ of $\Z[q]$-modules, compatible with the Frobenius and Verschiebung maps on big $q$-Witt vectors (Construction~\ref{constructionqfrobeniiandverschiebungen}), and such that $F_{m'/m} \circ V_{m'/m} = m'/m$ and $V_{m'/m} \circ F_{m'/m} = [m'/m]_{q^m}$.
    \end{enumerate}
\end{remark}

\begin{definition}[Canonical maps]\label{definitioncanonicalmapsqWittvectors}
    Let $R$ be a commutative ring, and $d \ge 1$ be an integer. For every integer $m \ge 1$, the \emph{canonical map}
    $$\can : \qWm(R) \longrightarrow \qWm^{(d)}(R)$$
    is the morphism of commutative $\Z[q]$-algebras induced by the inclusion of ideals $\mathbb{I}_m \subseteq \mathbb{I}_m^{(d)}$ of $\W_m(R)[q]$ (Definitions~\ref{definitionqbigWittvectors} and~\ref{definitionbigqWittvectorsatd}).
\end{definition}

\begin{remark}
    Let $R$ be a commutative ring, and $d \ge 1$ be an integer. Any $q$-FV-system at $d$ over~$R$ is in particular a $q$-FV-system over $R$, and the canonical maps of Definition~\ref{definitioncanonicalmapsqWittvectors} are induced by the universal property of $(\qWm(R))_{m \ge 1}$ (Remarks~\ref{remarkuniversalpropertyqWittvectors}).
\end{remark}

\begin{remark}[$q$-ghost maps at $d$]
    Let $R$ be a commutative ring, and $d \ge 1$ be an integer. Following Construction~\ref{constructioncyclotomicqFVsystem} and Definition~\ref{definitionqghostmaps}, there is a natural $q$-ghost map 
    $$q\text{-}\gh : \qWm^{(d)}(R) \longrightarrow \prod_{d|e|m} R[q]/\Phi_e(q), \quad c \longmapsto (c_e(q))_{d|e|m}$$
    on big $q$-Witt vectors at $d$ for every integer $m \ge 1$, which is compatible with the canonical map of Definition~\ref{definitioncanonicalmapsqWittvectors}.
\end{remark}

\begin{proposition}[Cyclotomic Frobenii]\label{propositioncyclotomicFrobeniiqWittvectors}
    Let $R$ be a commutative ring, and $d \ge 1$ be an integer. Then for every integer $m \ge 1$, there exists a natural morphism of commutative rings
    $$\Frob_d^{\Cyc} : \qWm(R) \longrightarrow q\text{-}\W_{dm}^{(d)}(R)$$
    given by $(c_e(q))_{e|m} \mapsto (c'_e(q))_{d|e|dm} := (c_{e/d}(q^d))_{d|e|dm}$ in $q$-ghost coordinates. This morphism is uniquely determined by the fact that it is functorial in $R$ and by the previous formula on $q$-ghost coordinates.
\end{proposition}

\begin{proof}
    First note that the morphism of commutative rings $\prod_{e|m} R[q]/\Phi_e(q) \rightarrow \prod_{d|e|dm} R[q]/\Phi_e(q)$ given by $(c_e(q))_{e|m} \mapsto (c_{e/d}(q^d))_{d|e|dm}$ is well-defined, because $\Phi_e(q)$ divides $\Phi_{e/d}(q^d)$ in $\Z[q]$.
    
    If $R$ is a polynomial $\Z$-algebra, then $R$ admits lifts of Frobenius $\varphi_p : R \rightarrow R$ for all prime numbers~$p$. In this case, we use the $q$-Dwork lemma (Lemma~\ref{lemmaqDwork})\footnote{Note that the statement and the proof of the $q$-Dwork lemma also hold for big $q$-Witt vectors at $d$.} to prove the desired result. To do so, let $(c_e(q))_{e|m} \in \prod_{e|m} R[q]/\Phi_e(q)$ be an element in the image of the $q$-ghost map, and $(\tilde{c}_e(q))_{e|m} \in \prod_{e|m} R[q]$ be a lift of this element such that, for every prime number $p$ and every integer $e \ge 1$ satisfying $pe|m$, we have $\tilde{c}_e(q) \equiv \varphi_p(\tilde{c}_{pe}(q))$ in $(R/p^{v_p(m/e)})[q]$. 
    
    It suffices to prove that, for every prime number $p$ and every integer $e \ge 1$ satisfying $d|e$ and $pe|dm$, we have $\tilde{c}_{e/d}(q^d) \equiv \varphi_p(\tilde{c}_{pe/d}(q^d))$ in $(R/p^{v_p(dm/e)})[q]$. This congruence holds as a consequence of the assumption on $(\tilde{c}_e(q))_{e|m}$, where we use that $p^{v_p(m/(e/d))}=p^{v_p(dm/e)}$.  
    
    If $R$ is a general commutative ring, then arguing as in the second part of the proof of Proposition~\ref{prop:q-angeltveit} implies the desired result.
\end{proof}

The following variant of Definition~\ref{definitionbigqWittvectorsatd} will be used in Section~\ref{sec:first-polylog}.

\begin{variant}[Cyclotomic ring at $d$]\label{variantcyclotomicringatd}
    Let $R$ be a commutative ring, and $d \ge 1$ be an integer. For every integer $m \ge 1$, the \emph{$m$-truncated cyclotomic ring of $R$ at $d$} is the commutative $\Z[q]$-algebra
    $$\mathcal{C}_{R,m}^{(d)} := \qWm^{(d)}(R)\big[\Phi_e(q)^{-1}~\big|~d \nmid e>1\big]$$
    where the localisation is indexed by integers $e > 1$ satisfying $d \nmid e$ and $e|m$. In particular, the morphisms of Definition~\ref{definitioncanonicalmapsqWittvectors} and Proposition~\ref{propositioncyclotomicFrobeniiqWittvectors} on big $q$-Witt vectors induce natural morphisms
    $$\can : \CRm \longrightarrow \CRm^{(d)} \quad \quad \Frob_d^{\Cyc} : \CRm \longrightarrow \mathcal{C}_{R,dm}^{(d)}$$
    of commutative rings, which we call the \emph{canonical map} and the \emph{$d^{th}$ cyclotomic Frobenius}. Similarly, the \emph{cyclotomic ring of $R$ at $d$} is the commutative $\Z[q]$-algebra $\CR^{(d)}$ defined as the limit of the truncated cyclotomic rings $\CRm^{(d)}$ over integers $m \ge 1$ satisfying $(m,N)=1$, along the Frobenius maps $F_{m'/m}$ of Remark~\ref{remarkuniversalpropertyqWittvectorsatd}. Note in particular that $\CR^{(d)}=0$ if $(d,N)>1$.
\end{variant}

The map $\Frob_d^{\Cyc}$ in Variant~\ref{variantcyclotomicringatd} is called \emph{cyclotomic Frobenius} (despite the map $F_d$ in Constructions~\ref{constructionfrobeniiandverschiebungen} and~\ref{constructionqfrobeniiandverschiebungen} already being called Frobenius) because of the following remarks, which states that $\Frob_d^{\Cyc}$ behaves as a ``universal $q$-lift of Frobenius'', at least when $R$ can be equipped with a $\Lambda$-ring structure.

\begin{remark}\label{remarkLambdaring}
    Let $R$ be a $\Lambda$-ring with Adams operations $(\psi^n)_{n \ge 1}$, and $d \ge 1$ be an integer. For every integer $m \ge 1$, there are natural commutative diagrams of commutative rings
    $$\begin{tikzcd}
    \qW_m(R) \arrow[r,"c_m"]\arrow[d,"\mathrm{can}"'] & R[q]/(q^m-1) \arrow[d,"\mathrm{can}"] \\ 
    \qW_{m}^{(d)}(R) \arrow[r,"c_m^{(d)}"] & R[q]/\prod_{d|e|m}{\Phi_e(q)}
    \end{tikzcd}
    \qquad 
    \begin{tikzcd}
    \qW_m(R) \arrow[r,"c_m"]\arrow[d,"\Frob^{\Cyc}_d"'] & R[q]/(q^m-1)\arrow[d,"\psi^d\otimes (q\mapsto q^d)"]\\ 
    \qW_{dm}^{(d)}(R) \arrow[r,"c_{dm}^{(d)}"] & R[q]/\prod_{d|e|dm}{\Phi_e(q)}
    \end{tikzcd}$$
    where the top horizontal maps are defined in \cite[Lemma~2.34]{wagner_q-witt_2024}, the bottom horizontal maps are defined as in \cite[Lemma~2.34]{wagner_q-witt_2024}, and the left vertical maps are defined in Definition~\ref{definitioncanonicalmapsqWittvectors} and Proposition~\ref{propositioncyclotomicFrobeniiqWittvectors}. In particular, there are similar commutative diagrams for the cyclotomic rings of Definition~\ref{definitioncyclotomicring} and Variant~\ref{variantcyclotomicringatd}, and these commutative diagrams are compatible with the isomorphism of Examples~\ref{examplesqWittvectors}\,(1) when the $\Lambda$-ring $R$ is perfect.
\end{remark}

\section{The cyclotomic logarithm}\label{sec:cyclotomic-logarithm}

Let $K$ be a number field, with ring of integers $\mathcal{O}_K$ and discriminant $\Delta_K$. In this section, we define the \emph{(derived) cyclotomic logarithm}
$$\dlog_{\Cyc} : \G_m(R)[-1] \longrightarrow \mathcal{N}^{\ge 1} \mathcal{C}_R\{1\}$$
of $R:=\nbrrg$ (Definition~\ref{definitioncyclotomiclogarithm}). Our construction is a global analogue of Bhatt--Lurie's prismatic logarithm \cite[Section~2]{bhatt_absolute_2022}. To bypass the use of quasisyntomic descent used in their construction, we follow Mao's alternative approach to the prismatic logarithm \cite[Section~2]{mao_prismatic_2024}.

\subsection{Frobenius maps on the Habiro ring of a number field}\label{subsectionHabirorings}

In this subsection, we recall the definition of the Habiro rings from the recent work of Garoufalidis--Scholze--Wheeler--Zagier \cite{garoufalidis_habiro_2024}, and construct an analogue of the cyclotomic Frobenius maps of the previous section in this context (Construction~\ref{constructionHabiroFrobenii}).

\begin{definition}[Habiro ring, \cite{garoufalidis_habiro_2024,wagner_q-hodge_2025}]\label{def:habiro}
    Let $R$ be an étale $\Z$-algebra and, for every prime number~$p$, let $\varphi_{R^\wedge_p}$ be the unique lift of Frobenius of the formally étale $\Z_{p}$\nobreakdash-algebra $R^\wedge_{p}$. 
    The \emph{Habiro ring of} $R$ is the commutative ring
    \[
    \mathcal{H}_{R}:=\left\{(f_e(q))_{e \ge 1} \in \prod_{e \ge 1}{R[q]^\wedge_{\Phi_e(q)}}~\bigg |~f_e(q) = \varphi_p(f_{p e}(q))~\text{in}~R[q]^\wedge_{(\Phi_{e}(q),\Phi_{pe}(q))}\right\}
    \]
    where the condition is taken over all prime numbers $p$ and integers $e \ge 1$, and where the endomorphism $\varphi_{p} : R[q]^\wedge_{(\Phi_{e}(q),\Phi_{pe}(q))} \rightarrow R[q]^\wedge_{(\Phi_{e}(q),\Phi_{pe}(q))}$ is defined as the $(p,\Phi_e(q))$-completion\footnote{Here we use the equality of ideals $(\Phi_e(q),\Phi_{pe}(q)) = (p,\Phi_e(q))$ in the commutative ring $\Z[q]$ (\cite[Lemma~2.1]{wagner_q-witt_2024}).} of $\varphi_{R^\wedge_p} \otimes \text{id} : R^\wedge_p \otimes_{\Z} \Z[q] \rightarrow R^\wedge_p \otimes_{\Z} \Z[q]$.
\end{definition}

We will mostly use the following truncated variant of Definition~\ref{def:habiro}, which already appears in \cite[2.7]{wagner_q-hodge_2025}.

\begin{variant}[Truncated Habiro ring]\label{varianttruncatedHabiroring}
Let $R$ be an étale $\Z$-algebra, and $m\geq 1$ be an integer. The \emph{$m$-truncated Habiro ring of $R$} is the commutative $\Z[q]$-algebra 
\[
    \mathcal{H}_{R,m}:=\left\{(f_e(q))_{e|m} \in \prod_{e|m}{R[q]^\wedge_{\Phi_e(q)}}~\bigg |~f_e(q) = \varphi_p(f_{p e}(q))~\text{in}~R[q]^\wedge_{(\Phi_{e}(q),\Phi_{pe}(q))}\right\}
\]
where the condition is taken over all prime numbers $p$ and integers $e \ge 1$ such that $pe|m$. 
\end{variant}

\begin{remark}
    For every étale $\Z$-algebra $R$, we have $\mathcal{H}_R = \lim_{m \ge 1} \mathcal{H}_{R,m}$.
\end{remark}

\begin{variant}[Truncated Habiro ring at $d$]\label{variantHabiroringatd} 
    Let $R$ be an étale $\Z$-algebra, 
    and $d\geq 1$ be an integer. For every integer $m \ge 1$, the \emph{$m$-truncated Habiro ring of $R$ at $d$} is the commutative $\Z[q]$-algebra
    $$\mathcal{H}_{R,m}^{(d)}:=\left\{(c_e)_{d|e|m} \in \prod_{d|e|m}{R[q]^\wedge_{\Phi_e(q)}}~\bigg |~c_e = \varphi_p(c_{p e})~\text{in}~R[q]^\wedge_{(\Phi_{e}(q),\Phi_{pe}(q))}\right\}\left[\Phi_e(q)^{-1}~\big|~d \nmid e>1\right]$$
    where the condition is taken over all prime numbers $p$ and integers $e \ge 1$ such that $d|e$, and the localisation is indexed by integers $e>1$ satisfying $d \nmid e$ and $e|m$.
\end{variant}

\begin{remark}
    Inverting the cyclotomic polynomials $\Phi_e(q)$ for $d \nmid e > 1$ in Variant~\ref{variantHabiroringatd} seems to be necessary to define the cyclosyntomic first Chern class in Section~\ref{subsec:the-first-chern-class}.
\end{remark}

\begin{remark}\label{remarkHabiroqWitt}
    Let $R$ be an \'etale $\Z$-algebra. For every integer $m \ge 1$, it was proved by Wagner that there is a natural isomorphism of commutative $\Z[q]$-algebras $\mathcal{H}_{R,m}/(q^m-1) \cong \qWm(R)$ (\cite[Theorem~$2.9$]{wagner_q-hodge_2025}). Moreover, for every integer $m' \ge 1$ satisfying $m|m'$, this identification induces a natural commutative diagram of $\Z[q]$-algebras 
    \[
    \begin{tikzcd}
    \mathcal{H}_{R,m'} \arrow[r]\arrow[d,"\text{can}"] & q\text{-}\mathbb{W}_{m'}(R) \arrow[d,"F_{m'/m}"] \\
    \mathcal{H}_{R,m} \arrow[r] & \qWm(R)
    \end{tikzcd}
    \]
    where the left vertical map is the natural projection map and the right vertical map is the Frobenius map on big $q$-Witt vectors (\cite[Remark~2.10]{wagner_q-hodge_2025}). 
\end{remark}

\begin{construction}[Habiro Frobenii]\label{constructionHabiroFrobenii}
    Let $R$ be an étale $\Z$-algebra, and $d \ge 1$ be an integer. The \emph{$d^{\text{th}}$ Frobenius on the Habiro ring} $\mathcal{H}_{R,m}$ is the ring homomorphism
    $$\text{Frob}_d^{\text{Hab}} : \mathcal{H}_{R,m} \longrightarrow \mathcal{H}_{R,dm}^{(d)}$$
    given by $(f_e(q))_{e|m} \mapsto (f'_{e}(q))_{d|e|dm} :=(f_{e/d}(q^d))_{d|e|dm}$. More precisely, this map is induced by the composite map
    $$\displaystyle\prod_{e|m} R[q]^\wedge_{\Phi_e(q)} \longrightarrow \displaystyle\prod_{d|e|dm} R[q]^\wedge_{\Phi_{e/d}(q^d)} \longrightarrow \displaystyle\prod_{d|e|dm} R[q]^\wedge_{\Phi_{e}(q)}$$
    where the first map is given by $(f_e(q))_{e|m} \mapsto (f_{e/d}(q^d))_{d|e|dm}$ and the second map is the $\Phi_{e}(q)$\nobreakdash-com\-pletion on each factor.\footnote{Note that this second map is well-defined as a consequence of the fact that $\Phi_{e}(q)$ divides $\Phi_{e/d}(q^d)$ in $\Z[q]$.} As a consequence of the series of equalities
    $$f'_{e}(q) = f_{e/d}(q^d) = \varphi_p(f_{pe/d}(q^d)) = \varphi_p(f'_{pe}(q))$$
    in $R[q]^{\wedge}_{(\Phi_{e}(q),\Phi_{pe}(q))}$, the image of $\mathcal{H}_{R,m}$ by this composite map is indeed contained in $\mathcal{H}_{R,dm}^{(d)}$.
\end{construction}

\begin{remark}\label{remarkpadicrealisationHabiroandcyclotomicFrobenius}
    Let $R$ be an étale $\Z$-algebra, $m \ge 1$ and $r \ge 0$ be integers, and $p$ be a prime number. 
    Following Lemma~\ref{lemmapq-1completion} (and its proof) and Remark~\ref{remarkpadicrealisationcyclotomic}, there is a natural morphism of commutative $\Z[q]$-algebras $\mathcal{C}_{pm}^{(p)} \rightarrow R[q]/(p,q-1)^r$. This morphism fits naturally in the commutative diagram of commutative rings
    $$\begin{tikzcd}
        \mathcal{H}_{R,m} \ar[d,"\Frob_p^{\Hab}"] \ar[r] & \mathcal{C}_{R,m} \ar[r] \ar[d,"\Frob_p^{\Cyc}"] & R[q]/(p,q-1)^r \ar[d,"\Frob_p^{\Prism}"] \\
        \mathcal{H}_{R,pm}^{(p)} \ar[r] & \mathcal{C}_{R,pm}^{(p)} \ar[r] & R[q]/(p,q-1)^r
    \end{tikzcd}$$
    where the right vertical map is defined as the unique morphism of commutative rings induced by~$\varphi_p$ on $R/p^r$ and sending $q$ to $q^p$. Here we use that $\Phi_e(q)$ is invertible in $R[q]/(p,q-1)^r$ for every integer $e \ge 1$ satisfying $p \nmid e$.
\end{remark}

\begin{lemma}\label{lemmaqghostmapinjectiveHabiro}
    Let $R$ be an étale $\Z$-algebra, and $m \ge 1$ be an integer. Then for any sequence of nonnegative integers $(n_e)_{e|m}$, the morphism of commutative $\Z[q]$-algebras
    $$q\text{-}\gh : \mathcal{H}_{R,m}/\prod_{e|m} \Phi_e(q)^{n_e} \longrightarrow \prod_{e|m} R[q]/\Phi_e(q)^{n_e}, \quad f \longmapsto (f_e(q))_{e|m}$$
    is injective.
\end{lemma}

\begin{proof}
    Because the cyclotomic polynomials $\Phi_e(q)$ are irreducible in the unique factorisation domain $\Z[q]$, the natural morphism of commutative rings
    $$\Z[q]/\prod_{e|m} \Phi_e(q)^{n_e} \longrightarrow \prod_{e|m} \Z[q]/\Phi_e(q)^{n_e}$$
    is injective. Moreover, the quotient of the Habiro ring $\mathcal{H}_{R,m}/(q^m-1)^n$ is étale over the commutative ring $\Z[q]/(q^m-1)^n$ for every integer $n \ge 0$ (\cite[Theorem~2.9]{wagner_q-hodge_2025}). In particular, the commutative ring $\mathcal{H}_{R,m}/\prod_{e|m} \Phi_e(q)^{n_e}$ is flat over $\Z[q]/\prod_{e|m} \Phi_e(q)^{n_e}$. Tensoring the previous \hbox{injective} morphism with the flat $\Z[q]/\prod_{e|m} \Phi_e(q)^{n_e}$-module $\mathcal{H}_{R,m}/\prod_{e|m} \Phi_e(q)^{n_e}$ 
    then induces the desired injective morphism of commutative $\Z[q]$-algebras 
    $$\mathcal{H}_{R,m}/\prod_{e|m} \Phi_e(q)^{n_e} \longrightarrow \prod_{e|m} R[q]/\Phi_e(q)^{n_e}$$
    where we use the canonical identification $(\mathcal{H}_{R,m})^\wedge_{\Phi_e(q)} \cong R[q]^\wedge_{\Phi_e(q)}$ (\cite[2.7]{wagner_q-hodge_2025}).
\end{proof}

\begin{corollary}\label{corollaryinjectiveHabiromandm'}
    Let $R$ be an étale $\Z$-algebra. Then for any integers $m,m' \ge 1$ satisfying $m|m'$, the natural morphism of commutative $\Z[q]$-algebras
    \begin{equation}\label{eq:cor-mod-(qm-1)(qm'-1)}\mathcal{H}_{R,m'}/(q^m-1)(q^{m'}-1) \longrightarrow \mathcal{H}_{R,m}/(q^m-1)^2 \times \prod_{\substack{e|m' \\ e \nmid m}} R[q]/\Phi_e(q)
    \end{equation}
    is injective.
\end{corollary}

\begin{proof}
    The post-composition of the morphism \eqref{eq:cor-mod-(qm-1)(qm'-1)} with the morphism of $\Z[q]$-algebras
    $$ \mathcal{H}_{R,m}/(q^m-1)^2 \times \prod_{\substack{e|m' \\ e \nmid m}} R[q]/\Phi_e(q) \longrightarrow \prod_{e|m} R[q]/\Phi_e(q)^2 \times \prod_{\substack{e|m' \\ e \nmid m}} R[q]/\Phi_e(q)$$
    is injective by Lemma~\ref{lemmaqghostmapinjectiveHabiro}, so the morphism \eqref{eq:cor-mod-(qm-1)(qm'-1)} is also injective.
\end{proof}

\subsection{The cyclotomic logarithm}\label{subsectioncyclotomiclogarithm}

In this subsection, we use the Habiro ring of Section~\ref{subsectionHabirorings} and (a refinement of) the norms on big $q$-Witt vectors of Section~\ref{subsectionnormsonqWittvectors} to construct the cyclotomic logarithm of an étale $\Z$-algebra (Construction~\ref{constructiontruncatedcyclotomiclogarithm} and Definition~\ref{definitioncyclotomiclogarithm}).

\begin{lemma}\label{lemmadivisionbym/e}
    Let $m,m' \ge 1$ be integers satisfying $m|m'$.
    \begin{enumerate}
        \item\label{item:divisible-image} The canonical map of $\Z[q]$-modules $(q^{m'}-1)/(q^{m'}-1)^2 \rightarrow (q^m-1)/(q^m-1)^2$ is divisible by~$\frac{m'}{m}$.
        \item\label{item:divided-image} The induced map $(\frac{m'}{m})^{-1} : (q^{m'}-1)/(q^{m'}-1)^2 \rightarrow (q^m-1)/(q^m-1)^2$ is surjective.
        \item\label{item:isomorphism} In particular, the previous map induces an isomorphism  
        $$(\tfrac{m'}{m})^{-1} : (q^{m'}-1)/(q^m-1)(q^{m'}-1) \xrightarrow{\cong} (q^m-1)/(q^m-1)^2$$ of $\Z[q]/(q^m-1)$-modules.
    \end{enumerate}
\end{lemma}

\begin{proof}
    This is a consequence of the congruence
    $$\frac{q^{m'}-1}{q^m-1} =\left[\frac{m'}{m}\right]_{q^m}\equiv \frac{m'}{m} \pmod{q^m-1}$$
    in the commutative ring $\Z[q]$, and of the fact that the module $(q^m-1)/(q^m-1)^2$ is flat over $\Z$.
\end{proof}

\begin{definition}[Nygaard twist]\label{definitionNygaardtwistedcyclotomicring}
    Let $R$ be a commutative ring. For every integer $m \ge 1$, the \emph{$m$-truncated Nygaard twist of $R$} is the invertible $\qWm(R)$-module
    $$\mathcal{N}^{\ge 1} \CRm\{1\} := \qWm(R) \otimes_{\Z[q]} (q^m-1)/(q^m-1)^2.$$
    If $K$ is a number field and $R$ is the étale $\Z$-algebra $\nbrrg$, the \emph{Nygaard twist of $R$} is the $\Z[q]$\nobreakdash-module
    $$\mathcal{N}^{\ge 1} \CR\{1\} := \lim_{\substack{m \ge 1 \\ (m,N)=1}} \mathcal{N}^{\ge 1} \CRm\{1\}$$
    where the limit is taken over integers $m \ge 1$ satisfying $(m,N)=1$,\footnote{The restriction $(m,N)=1$ is not necessary in this section to define the cyclotomic logarithm. We use this convention because it will be necessary to restrict to these integers $m \ge 1$ in order to define the first $q$-polylogarithm class of a cyclotomic unit in the next section (Definition~\ref{definitionpolylogarithmcyclosyntomic}).} and along the maps induced by the Frobenius maps $F_{m'/m}$ on $q$-Witt vectors (Construction~\ref{constructionqfrobeniiandverschiebungen}) and the maps $(\frac{m'}{m})^{-1}$ on the second factor (Lemma~\ref{lemmadivisionbym/e}\,(2)).
\end{definition}

\begin{remark}\label{remarkNygaardwithHabiro}
    Let $R$ be an étale $\Z$-algebra. For every integer $m \ge 1$, there is a natural isomorphism of $\Z[q]$-modules
    $$\mathcal{N}^{\ge 1} \CRm\{1\} \cong (q^m-1)\mathcal{H}_{R,m}/(q^m-1)^2\mathcal{H}_{R,m}$$
    where the Frobenius maps big on $q$-Witt vectors correspond to the canonical restriction maps on Habiro rings (Remark~\ref{remarkHabiroqWitt}).
\end{remark}

\begin{construction}[Derived logarithm]\label{constructiondlog}
    Given a commutative ring $A$ and an ideal $I \subseteq A$, there is a natural short exact sequence of abelian groups
    $$0 \longrightarrow I/I^2 \xlongrightarrow{\text{exp}} \G_m(A/I^2) \longrightarrow \G_m(A/I) \longrightarrow 0$$
    where $\text{exp}$ is given by $x \mapsto 1+x$. This induces a boundary map $$\dlog_{(A,I)} : \G_m(A/I) \longrightarrow I/I^2[1]$$ in the derived category $\mathcal{D}(\Z)$ (see also \cite[Construction 2.7]{mao_prismatic_2024} for a generalisation to animated rings).
\end{construction}

\begin{construction}[Truncated cyclotomic logarithm]\label{constructiontruncatedcyclotomiclogarithm}
    Let $R$ be an étale $\Z$-algebra, and $m \ge 1$ be an integer. The \emph{m-truncated cyclotomic logarithm of} $R$ is the map
    $$\dlog_{\Cyc}^{(m)} : \G_m(R)[-1] \longrightarrow \mathcal{N}^{\ge 1}\CRm\{1\}$$
    in the derived category $\mathcal{D}(\Z)$ defined as the composite
    $$\G_m(R)[-1] \xlongrightarrow{\Pi_m} \G_m(\qW_m(R))[-1] \xlongrightarrow{\dlog} \mathcal{N}^{\ge 1}\CRm\{1\}$$
    where the first map is the cyclotomic norm of Proposition~\ref{prop:q-angeltveit} and the second map is the derived logarithm of Construction~\ref{constructiondlog} for the pair $(A,I) := (\mathcal{H}_{R,m},(q^m-1))$, where $\mathcal{H}_{R,m}$ and $\mathcal{N}^{\ge 1}\CRm\{1\}$ are respectively defined in Variant~\ref{varianttruncatedHabiroring} and in Definition~\ref{definitionNygaardtwistedcyclotomicring}. Here we use the identification $\mathcal{H}_{R,m}/(q^m-1)\cong\qWm(R)$ (Remark~\ref{remarkHabiroqWitt}).
\end{construction}

In the rest of this subsection, we construct the cyclotomic logarithm
$$\dlog_{\Cyc} : \G_m(R)[-1] \longrightarrow \mathcal{N}^{\ge 1} \CR\{1\}$$
by proving that the truncated cyclotomic logarithms of Construction~\ref{constructiontruncatedcyclotomiclogarithm} are compatible, in the derived category $\mathcal{D}(\Z)$, between different integers $m \ge 1$. To do so, we construct compatible homotopies $h_{m,m'}$ making the suitable transition diagrams commute (Proposition~\ref{propositioncompatiblehomotopies}). 
The homotopies $h_{m,m'}$ are in turn constructed out of a partial Habiro lift $\tilde{\Pi}_{m'/m}$ of the cyclotomic norms $\Pi_{m'/m}$ of Section~\ref{subsectionnormsonqWittvectors}, which we call the \emph{lifted cyclotomic norms} (Proposition~\ref{propositionliftedcyclotomicpower} and Remark~\ref{remarkHabirologarithmcannotexist}).

\begin{lemma}\label{lemmadiagramchase}
    Let 
    \[\begin{tikzcd}
	& {B_0} & {C_0} \\
	{A_1} & {B_1} & {C_1} \\
	{A_2} & {B_2} & {C_2}
	\arrow["{\beta_0}", from=1-2, to=1-3]
	\arrow["{f_B}"', from=1-2, to=2-2]
	\arrow["{f_C}"', from=1-3, to=2-3]
	\arrow["{h_C}", shift left=5, bend right=-30pt, from=1-3, to=3-3]
	\arrow["{\alpha_1}", hook, from=2-1, to=2-2]
	\arrow["{\beta_1}", two heads, from=2-2, to=2-3]
	\arrow["{g_A}", two heads, from=3-1, to=2-1]
	\arrow["{\alpha_2}", hook, from=3-1, to=3-2]
	\arrow["{g_B}", from=3-2, to=2-2]
	\arrow["{\beta_2}", two heads, from=3-2, to=3-3]
	\arrow["{g_C}", from=3-3, to=2-3]
    \end{tikzcd}\]
    be a commutative diagram of abelian groups whose middle and bottom lines are short exact sequences. If the morphism $g_A$ is surjective, then the image of $f_B$ is contained in the image of $g_B$. In particular, if the morphism $g_B$ is injective, then there exists a unique morphism of abelian groups $h_B : B_0 \rightarrow B_2$ satisfying $f_B=g_B \circ h_B$.
\end{lemma}

\begin{proof}
    It suffices to prove the inclusion $f_B(B_0)\subseteq g_B(B_2)$ as subsets of $B_1$. We prove first that $f_B(B_0) \subseteq \beta_1^{-1}(g_C(C_2))$ and then that $\beta_1^{-1}(g_C(C_2)) \subseteq g_B(B_2)$. The first inclusion follows from the existence of the morphism $h_C$: given an element $b_0 \in B_0$, we have
    $$\beta_1(f_B(b_0)) = f_C(\beta_0(b_0)) = g_C(h_C(\beta_0(b_0))) \in g_C(C_2).$$
    The second inclusion follows the snake lemma and the surjectivity of the morphism $g_A$.
\end{proof}

\begin{proposition}[Lifted cyclotomic norms]\label{propositionliftedcyclotomicpower}
    Let $R$ be an étale $\Z$-algebra. Then for any integers $m,m' \ge 1$ satisfying $m|m'$, the morphism of multiplicative monoids
    \begin{equation}\label{eq:Pi-tilde}
    \tilde{\Pi}_{m'/m} : \prod_{e|m} R[q]/\Phi_e(q)^2 \longrightarrow \prod_{e|m} R[q]/\Phi_e(q)^2 \times \prod_{\substack{e|m' \\ e \nmid m}} R[q]/\Phi_e(q)
    \end{equation}
    given by $(f_e(q))_{e|m} \mapsto (f'_e(q))_{e|m'} := \big(f_{(e,m)}(q^{\frac{e}{(e,m)}})^{\frac{m'}{[e,m]}}\big)_{e|m'}$ is well-defined, and restricts to a morphism of multiplicative monoids
    $$\tilde{\Pi}_{m'/m} : \mathcal{H}_{R,m}/(q^m-1)^2 \longrightarrow \mathcal{H}_{R,m'}/(q^m-1)(q^{m'}-1).$$
\end{proposition}

\begin{proof}
    The fact that the morphism \eqref{eq:Pi-tilde} is well-defined is a consequence of the fact that, for every integer $e \ge 1$ satisfying $e|m'$, we have $\Phi_e(q)^2|\Phi_{(e,m)}(q^{\frac{e}{(e,m)}})^2$ in $\Z[q]$. It is in addition multiplicative by construction. 
    
    In the rest of the proof, we use the injectivity of Lemma~\ref{lemmaqghostmapinjectiveHabiro} to interpret the commutative $\Z[q]$\nobreakdash-algebras \hbox{$\mathcal{H}_{R,m}/(q^m-1)^2$} and $\mathcal{H}_{R,m'}/(q^m-1)(q^{m'}-1)$ as subrings of the source and target of the morphism $\eqref{eq:Pi-tilde}$. 
    To prove that the image of $\mathcal{H}_{R,m}/(q^m-1)^2$ by the morphism \eqref{eq:Pi-tilde} is contained in $\mathcal{H}_{R,m'}/(q^m-1)(q^{m'}-1)$, first note that if $e|m$, then $f'_e(q) = f_e(q)^{\frac{m'}{m}}$, so that $(f'_e(q))_{e|m}$ defines an element of $\mathcal{H}_{R,m}/(q^m-1)^2$. 
    Moreover, we have that $\mathcal{H}_{R,m'}/(q^m-1)^2 \cong \mathcal{H}_{R,m}/(q^m-1)^2$ by \cite[2.7]{wagner_q-hodge_2025}, hence the isomorphism of commutative $\Z[q]$-algebras $$\mathcal{H}_{R,m}/((q^m-1)^2,(q^{m'}-1)) \cong \qWmm(R)/(q^m-1)^2$$ 
    by \cite[Theorem~$2.9$]{wagner_q-hodge_2025}.
    In particular, there is a commutative diagram of abelian groups
    
    \[\begin{tikzcd}[column sep=1.5 em]
	& {\frac{\mathcal{H}_{R,m}}{(q^m-1)^2\mathcal{H}_{R,m}}} & {\qW_m(R)} \\
	{\frac{(q^{m'}-1)\mathcal{H}_{R,m'}}{(q^{m}-1)(q^{m'}-1)\mathcal{H}_{R,m'}}} & \begin{array}{c} \frac{\mathcal{H}_{R,m}}{(q^m-1)^2\mathcal{H}_{R,m}}\times \prod_{\substack{e|m' \\e \nmid m}}{\frac{R[q]}{\Phi_e(q)}} \end{array} & \begin{array}{c} \frac{\qW_{m'}(R)}{(q^m-1)^2}\times \prod_{\substack{e|m' \\e \nmid m}}{\frac{R[q]}{\Phi_e(q)}} \end{array} \\
	{\frac{(q^{m'}-1)\mathcal{H}_{R,m'}}{(q^{m}-1)(q^{m'}-1)\mathcal{H}_{R,m'}}} & {\frac{\mathcal{H}_{R,m'}}{(q^{m}-1)(q^{m'}-1)\mathcal{H}_{R,m'}}} & {\qW_{m'}(R)}
	\arrow[two heads, from=1-2, to=1-3]
	\arrow["{\tilde{\Pi}_{m'/m}}"', from=1-2, to=2-2]
	\arrow["{\Pi_{m'/m}}"', from=1-3, to=2-3]
	\arrow["{\Pi_{m'/m}}", shift left=5, bend right=-90pt, from=1-3, to=3-3]
	\arrow[hook, from=2-1, to=2-2]
	\arrow[no head, from=2-1, to=3-1]
	\arrow[shift right, no head, from=2-1, to=3-1]
	\arrow[two heads, from=2-2, to=2-3]
	\arrow[hook, from=3-1, to=3-2]
	\arrow["{\can}",hook', from=3-2, to=2-2]
	\arrow[two heads, from=3-2, to=3-3]
	\arrow["{\can}", from=3-3, to=2-3]
    \end{tikzcd}\]
    whose middle and bottom lines are short exact sequences, and where the commutativity of the right part of the diagram is a consequence of Proposition~\ref{prop:q-angeltveit}. The natural morphism of abelian groups
    $$\mathcal{H}_{R,m'}/(q^m-1)(q^{m'}-1) \longrightarrow \mathcal{H}_{R,m}/(q^m-1)^2 \times \displaystyle\prod_{\substack{e|m' \\e \nmid m}} R[q]/\Phi_e(q)$$
    is moreover injective by Corollary~\ref{corollaryinjectiveHabiromandm'}. 
    So the desired result is a consequence of Lemma~\ref{lemmadiagramchase} applied to the previous diagram.
\end{proof}

\begin{corollary}\label{corollaryliftedcyclotomicpowercommutativediagram}
    Let $R$ be an étale $\Z$-algebra. Then for any integers $m,m' \ge 1$ satisfying $m|m'$, the natural diagram of abelian groups
    $$\begin{tikzcd}
        \mathcal{N}^{\ge 1}\CRm\{1\} \arrow[r,"\emph{exp}"] \arrow[d,"\frac{m'}{m}","\cong"'] & \G_m\big({\frac{\mathcal{H}_{R,m}}{(q^m-1)^2\mathcal{H}_{R,m}}}\big) \arrow[d,"\tilde{\Pi}_{m'/m}"] \\
        {\frac{(q^{m'}-1)\mathcal{H}_{R,m'}}{(q^{m}-1)(q^{m'}-1)\mathcal{H}_{R,m'}}} \arrow[r,"\emph{exp}"] & \G_m\big({\frac{\mathcal{H}_{R,m'}}{(q^m-1)(q^{m'}-1)\mathcal{H}_{R,m'}}}\big)
    \end{tikzcd}$$
    is commutative, where $\mathcal{N}^{\ge 1}\CRm\{1\}$ is identified with $(q^m-1)\mathcal{H}_{R,m}/(q^m-1)^2 \mathcal{H}_{R,m}$ (Remark~\ref{remarkNygaardwithHabiro}).
\end{corollary}

\begin{proof}
    By Lemma~\ref{lemmaqghostmapinjectiveHabiro}, it suffices to prove that $\tilde{\Pi}_{m'/m}\circ \exp$ and $\exp\circ \frac{m'}{m}$ agree as maps from $\prod_{e|m}{(\Phi_e(q))/(\Phi_e(q))^2}$ to $\prod_{e|m} R[q]/\Phi_e(q)^2 \times \prod_{e|m',e \nmid m} R[q]/\Phi_e(q)$. 
    Let $(\Phi_e(q)c_e(q))_{e|m}$ be an element of $\prod_{e|m}{(\Phi_e(q))/(\Phi_e(q)^2)}$, and $(c'_e(q))_{e|m'}$ be its image under the function $\tilde{\Pi}_{m,m'} \circ \text{exp}$. If $e|m$, then 
    $$c'_e(q) = \big(1+\Phi_e(q)c_e(q)\big)^{\frac{m'}{m}} \equiv 1 + \tfrac{m'}{m}\Phi_e(q)c_e(q) = \text{exp}\big(\tfrac{m'}{m}\Phi_e(q)c_e(q)\big)$$
    in $R[q]/\Phi_e(q)^2$. 
    And if $e \nmid m$, then the $e^{th}$ $q$-ghost component of the commutative $\Z[q]$-module $(q^{m'}-1)/(q^ m-1)(q^{m'}-1)$ is zero by $\Z[q]/(q^m-1)$-linearity, and 
    $$c'_e(q) = \big(1+\Phi_{(e,m)}(q^{\frac{e}{(e,m)}})c_{(e,m)}(q^{\frac{e}{(e,m)}})\big)^{\frac{m'}{[e,m]}} \equiv 1$$
    in $R[q]/\Phi_e(q)$, because $\Phi_e(q) | \Phi_{(e,m)}(q^{\frac{e}{(e,m)}})$.
\end{proof}

\begin{remark}\label{remarkHabirologarithmcannotexist}
    Contrary to what Proposition \ref{propositionliftedcyclotomicpower} may suggest, the lifted cyclotomic norm $\tilde{\Pi}_{m'/m}$ cannot be naturally lifted to a morphism of abelian groups
    $$\tilde{\Pi}_{m'/m} : \G_m\big(\mathcal{H}_{R,m}/(q^m-1)^2\big) \longrightarrow \G_m\big(\mathcal{H}_{R,m'}/(q^{m'}-1)^2\big).$$
    Indeed, in the case where $(m,m'):=(1,m)$, such a lift would induce a splitting of the extension 
    $[\dlog_{\Cyc}^{(m)}] \in \text{Ext}^1_{\Z}(\mathcal{N}^{\ge 1}\CRm\{1\},\G_m(\qWm(R)))$ 
    of Construction~\ref{constructiontruncatedcyclotomiclogarithm}, and in particular that the corresponding truncated cyclotomic logarithms are identically zero. However, we will prove that these cyclotomic logarithms are not identically zero in Section~\ref{sec:first-polylog}.
\end{remark}

\begin{proposition}\label{propositioncompatiblehomotopies}
    Let $R$ be an étale $\Z$-algebra. Then for every pair of positive integers $(m,m')$ satisfying $m|m'$, there exists an homotopy $h_{m,m'}$ making the diagram
    $$\begin{tikzcd}[column sep=large]
        \G_m(R)[-1] \arrow[d,"\emph{id}"] \arrow[r,"\dlog_{\Cyc}^{(m^\prime)}"] & \mathcal{N}^{\ge 1}\mathcal{C}_{R,m'}\{1\} \arrow[d,"F_{m'/m}\{1\}"] \\
        \G_m(R)[-1] \arrow[r,"\dlog_{\Cyc}^{(m)}"] & \mathcal{N}^{\ge 1} \CRm\{1\}
    \end{tikzcd}$$
    commute in the derived category $\mathcal{D}(\Z)$, and such that $h_{m',m''} \circ h_{m,m'} = h_{m,m''}$ for all integers $m,m',m'' \ge 1$ satisfying $m|m'|m''$.
\end{proposition}

\begin{proof} 
    Let $m$ and $m'$ be positive integers satisfying $m|m'$. We want to find an homotopy making the diagram
    $$\begin{tikzcd}
        \G_m(R) \arrow[d,"\text{id}"] \arrow[r,"\Pi_{m'}"] & \G_m(\qWmm(R)) \ar[r,"\dlog"] & \mathcal{N}^{\ge 1}\mathcal{C}_{R,m'}\{1\}[1] \arrow[d,"F_{m'/m}\{1\}"] \\
        \G_m(R) \arrow[r,"\Pi_m"] & \G_m(\qWm(R)) \ar[r,"\dlog"] & \mathcal{N}^{\ge 1}\mathcal{C}_{R,m}\{1\}[1]
    \end{tikzcd}$$
    commute in the derived category $\mathcal{D}(\Z)$. 
    By Remark~\ref{remarkcompatibilityupgradedcyclotomicpowers}, the diagram of abelian groups
    $$\begin{tikzcd} 
        \G_m(R) \arrow[d,"\text{id}"] \arrow[r,"\Pi_{m'}"] & \G_m(\qWmm(R)) \\
        \G_m(R) \arrow[r,"\Pi_m"] & \G_m(\qWm(R)) \ar[u,"\Pi_{m'/m}"]
    \end{tikzcd}$$
    is commutative, so it suffices to find an homotopy making the diagram
    $$\begin{tikzcd}
        \G_m(\qWmm(R)) \ar[r,"\dlog"] & \mathcal{N}^{\ge 1}\mathcal{C}_{R,m'}\{1\}[1] \arrow[d,"F_{m'/m}\{1\}"] \\
        \G_m(\qWm(R)) \ar[r,"\dlog"] \ar[u,"\Pi_{m'/m}"] & \mathcal{N}^{\ge 1}\CRm\{1\}[1]
    \end{tikzcd}$$
    commute in the derived category $\mathcal{D}(\Z)$. 
    By Definition~\ref{definitionNygaardtwistedcyclotomicring} and Remark~\ref{remarkNygaardwithHabiro}, the right vertical map of this diagram is naturally identified with (the shift of) the composite
    $$\mathcal{N}^{\ge 1} \mathcal{C}_{R,m'}\{1\} \twoheadrightarrow {\frac{(q^{m'}-1)\mathcal{H}_{R,m'}}{(q^{m}-1)(q^{m'}-1)\mathcal{H}_{R,m'}}} \xlongrightarrow{\cong} \mathcal{N}^{\ge 1} \CRm\{1\}$$
    where the first map is induced by the canonical map
    $$(q^{m'}-1)/(q^{m'}-1)^2 \twoheadrightarrow (q^{m'}-1)/(q^m-1)(q^{m'}-1)$$ of $\Z[q]$-modules, and the second map is induced by the isomorphism of Lemma~\ref{lemmadivisionbym/e}\,(3). 
    Moreover, the natural diagram
    $$\begin{tikzcd}
        \G_m(\qWmm(R)) \ar[r,"\dlog"] \ar[d,"\text{id}"] & \mathcal{N}^{\ge 1} \mathcal{C}_{R,m'}\{1\}[1] \arrow[d,two heads] \\
        \G_m(\qWmm(R)) \ar[r,"\dlog"] & {\frac{(q^{m'}-1)\mathcal{H}_{R,m'}}{(q^{m}-1)(q^{m'}-1)\mathcal{H}_{R,m'}}}[1]
    \end{tikzcd}$$
    in the derived category $\mathcal{D}(\Z)$, where the bottom horizontal map is the derived logarithm of Construction~\ref{constructiondlog} for the pair 
    $(A,I) := (\mathcal{H}_{R,m'}/(q^m-1)(q^{m'}-1),(q^{m'}-1))$, 
    is commutative.
    
    It then suffices to produce an homotopy making the diagram
    $$\begin{tikzcd}
        \G_m(\qWmm(R)) \ar[r,"\dlog"] & {\frac{(q^{m'}-1)\mathcal{H}_{R,m'}}{(q^{m}-1)(q^{m'}-1)\mathcal{H}_{R,m'}}}[1] \\
        \G_m(\qWm(R)) \ar[u,"\Pi_{m'/m}"] \ar[r,"\dlog"] & \mathcal{N}^{\ge 1}\CRm\{1\}[1] \ar[u,"\frac{m'}{m}","\cong"']
    \end{tikzcd}$$
    commute in the derived category $\mathcal{D}(\Z)$. Unwinding the definition of the derived logarithm (Construction~\ref{constructiondlog}), this amounts to finding a morphism of abelian groups $\tilde{\Pi}_{m'/m}$ making the diagram of abelian groups
    $$\begin{tikzcd}
        0 \ar[r] & {\frac{(q^{m'}-1)\mathcal{H}_{R,m'}}{(q^{m}-1)(q^{m'}-1)\mathcal{H}_{R,m'}}} \ar[r,"\exp"] & \G_m\big({\frac{\mathcal{H}_{R,m'}}{(q^{m}-1)(q^{m'}-1)\mathcal{H}_{R,m'}}}\big) \ar[r] & \G_m(\qWmm(R)) \ar[r] & 0 \\
        0 \ar[r] & \mathcal{N}^{\ge 1}\CRm\{1\} \ar[r,"\exp"] \ar[u,"\frac{m'}{m}"] & \G_m\big({\frac{\mathcal{H}_{R,m'}}{(q^{m'}-1)^2\mathcal{H}_{R,m'}}}\big) \ar[r] \ar[u,"\tilde{\Pi}_{m'/m}"] & \G_m(\qWm(R)) \ar[u,"\Pi_{m'/m}"] \ar[r] & 0
    \end{tikzcd}$$
    commute. Such morphisms $\tilde{\Pi}_{m'/m}$ are constructed in Proposition~\ref{propositionliftedcyclotomicpower}. More precisely, these morphisms are well-defined by Proposition~\ref{propositionliftedcyclotomicpower}, the right square is commutative by construction, and the left square is commutative by Corollary~\ref{corollaryliftedcyclotomicpowercommutativediagram}. Finally, the fact that the induced homotopies $h_{m,m'}$ satisfy $h_{m',m''} \circ h_{m,m'} = h_{m,m''}$ is a consequence of the formula $\tilde{\Pi}_{m''/m'} \circ \tilde{\Pi}_{m'/m} = \tilde{\Pi}_{m''/m}$, which can be checked as in the proof of Lemma~\ref{lemmanormpreserveimageofqghostmap}.
\end{proof}

\begin{definition}[Cyclotomic logarithm]\label{definitioncyclotomiclogarithm}
    Let $K$ be a number field, and $R$ be the étale $\Z$-algebra $\nbrrg$. 
    The \emph{cyclotomic logarithm of} $R$ is the map
    $$\dlog_{\Cyc} : \G_m(R)[-1] \longrightarrow \mathcal{N}^{\ge 1} \CR\{1\}$$
    defined as the inverse limit over integers $m \ge 1$ satisfying $(m,N)=1$, in the derived category~$\mathcal{D}(\Z)$, of the $m$-truncated cyclotomic logarithms of Construction~\ref{constructiontruncatedcyclotomiclogarithm}. Note here that we use the compatible homotopies of Proposition~\ref{propositioncompatiblehomotopies} to make sense of this inverse limit.\footnote{More precisely, to write down the limit of the maps $\dlog_{\Cyc}^{(m)}$ in the derived $\infty$-category $\mathcal{D}(\Z)$, we use the $1$\nobreakdash-homotopies $h_{m,m'}$ of Proposition~\ref{propositioncompatiblehomotopies} to encode the fact that the maps $\dlog_{\Cyc}^{(m)}$ are compatible between different integers $m \ge 1$, and the equality $h_{m',m''} \circ h_{m,m'} = h_{m,m''}$, also proved in Proposition~\ref{propositioncompatiblehomotopies}, to ensure that there are no higher coherences to check.}
\end{definition}

\begin{remark}[Comparison with \cite{mao_prismatic_2024}]\label{remarkcomparisontoMao}
    The general strategy to construct the cyclotomic logarithm $\dlog_{\Cyc}$ is inspired by Mao's construction of the refined prismatic logarithm $\dlog_{\Prism}$ \cite[Section~2]{mao_prismatic_2024}. However, note that Mao's $p$-adic constructions are limited to the case of odd prime numbers $p$. The explicit formula for the lifted cyclotomic norms (Proposition~\ref{propositionliftedcyclotomicpower}) allows us to bypass this technical assumption, and in particular not to restrict our constructions to odd integers~$m \ge 1$. 
\end{remark}

\section{Cyclosyntomic regulator}\label{sec:first-polylog}

Let $K$ be a number field. In this section, we define, for each integer $d \ge 2$, the \textit{cyclosyntomic complex}
$$R\Gamma_{\CycSyn}\big(R,\Z(1)^{(d)}\big) \in \mathcal{D}(\Z)$$
of $R:=\nbrrg$ (Definition~\ref{definitioncyclosyntomiccohomology}), its associated first Chern class
$$c_1^{\CycSyn} : R\Gamma_{\text{ét}}(R,\G_m)[-1] \longrightarrow R\Gamma_{\CycSyn}\big(R,\Z(1)^{(d)}\big)$$
as a map in the derived category $\mathcal{D}(\Z)$ (Definition~\ref{definitioncyclosyntomicfirstChernclass}), and prove that this cyclosyntomic first Chern class is computed at cyclotomic units by a $q$-deformation of the first polylogarithm (Theorem~\ref{theoremmainLi1cyclosyntomic}).

\subsection{The cyclosyntomic cohomology of a number field}\label{subsectionthecyclosyntomiccohomologyofanumberfield}

In this subsection, we introduce the cyclosyntomic cohomology of $R:=\mathcal{O}_K[\Delta_K^{-1}]$ (Definition~\ref{definitioncyclosyntomiccohomology}). It is defined as a two term complex given in degree zero by the Nygaard twist $\mathcal{N}^{\ge 1} \CR\{1\}$ (Definition~\ref{definitionNygaardtwistedcyclotomicring}) and in degree one by the cyclotomic twist $\CR^{(d)}\{1\}$, which we define now.

\begin{definition}[Cyclotomic twist at $d$]\label{definitiontwistedcyclotomicringatd}
    Let $R$ be a commutative ring, and $d \ge 1$ be an integer. For every integer $m \ge 1$, the \emph{$m$-truncated cyclotomic twist of $R$ at $d$} is the $\CRm^{(d)}$-module
    $$\CRm^{(d)}\{1\} := \CRm^{(d)} \otimes_{\Z[q]} (q^m-1)/(q^m-1)^2$$
    where the cyclotomic ring $\CRm^{(d)}$ is defined in Variant~\ref{variantcyclotomicringatd}. If $K$ is a number field and $R$ is the étale $\Z$-algebra $\nbrrg$, the \emph{cyclotomic twist of $R$ at $d$} is the $\Z[q]$-module
    $$\CR^{(d)}\{1\} := \lim_{\substack{m \ge 1 \\ (m,\Delta_K)=1}} \CRm^{(d)}\{1\}$$
    where the limit is taken over integers $m \ge 1$ satisfying $(m,N)=1$, and along the maps induced by the Frobenius maps $F_{m'/m}$ on $q$-Witt vectors (Construction~\ref{constructionqfrobeniiandverschiebungen}) and the maps $(\frac{m'}{m})^{-1}$ on the second factor (Lemma~\ref{lemmadivisionbym/e}\,(2)).
\end{definition}

\begin{remark}\label{remarklimitovermNygaardandtwistedcyclotomic}
    Let $R$ be an étale $\Z$-algebra, and $d \ge 2$ be an integer. For every integer $m \ge 1$, there are natural isomorphisms of $\Z[q]$-modules
    $$\CRm^{(d)}\{1\} \cong (q^m-1) \mathcal{H}^{(d)}_{R,m} / (q^m-1)^2\mathcal{H}^{(d)}_{R,m} \cong [m]_q\mathcal{H}^{(d)}_{R,m} / [m]_q^2\mathcal{H}^{(d)}_{R,m}$$
    where $\mathcal{H}_{R,m}^{(d)}$ is the $m$-truncated Habiro ring of $R$ at $d$ (Variant~\ref{variantHabiroringatd}), and where the Frobenius maps on $q$-Witt vectors correspond to the canonical restriction maps on Habiro rings (Remark~\ref{remarkHabiroqWitt}). Note that the first isomorphism also holds for $d=1$ and that, in general, one may replace $(q^m-1)$ and $[m]_q$ by $\prod_{d|e|m}\Phi_e(q)$.
\end{remark}

\begin{remark}
    When $d=p$ is a prime number and $m=p^r$ is a power of $p$, the identification of Remark~\ref{remarklimitovermNygaardandtwistedcyclotomic} in terms of $[m]_q$ is reminiscent of the notion of Breuil--Kisin twist in prismatic cohomology \cite[Section~2.2]{bhatt_absolute_2022}, in the sense that we have the equality of ideals
    $$I\,\varphi_A^\ast(I)\,\cdots\,(\varphi_A^{r-1})^\ast(I) = \big([p]_q [p]_{q^p} \cdots [p]_{q^{p^{r-1}}}\big) = ([p^r]_q)$$
    in the $q$-prism $(A,I) := (\Z_p[\![q-1]\!],[p]_q)$.
\end{remark}

\begin{notation}\label{notationcanonicalmapcyclosyntomic}
    Given a number field $K$, $R:=\mathcal{O}_K[\Delta_K^{-1}]$, and integers $d,m \ge 1$, we denote by
    $$\can : \mathcal{N}^{\ge 1}\CRm\{1\} \longrightarrow \CRm^{(d)}\{1\}$$
    the morphism of $\Z[q]$-modules given by $(\Phi_e(q)c_e(q))_{e|m} \mapsto (\Phi_e(q)c_e(q))_{d|e|m}$. These maps are compatible between different integers $m \ge 1$ by construction, so they induce a morphism of $\Z[q]$\nobreakdash-modules $$\can : \mathcal{N}^{\ge 1} \CR\{1\} \longrightarrow \CR^{(d)}\{1\}.$$
\end{notation}

The following definition is a variant of Construction~\ref{constructionHabiroFrobenii}.

\begin{construction}[Truncated twisted cyclotomic Frobenii]\label{constructiontruncatedtwistedcyclotomicFrobenius}
    Let $R$ be a commutative ring, and $d,m \ge 1$ be integers. The \emph{$m$-truncated $d^{th}$ twisted cyclotomic Frobenius} is the morphism of abelian groups
    $$\Frob_d^{\Cyc}\{1\} : \mathcal{N}^{\ge 1} \CRm\{1\} \longrightarrow \CRm^{(d)}\{1\}$$
    given by $(q^m-1)(c_e(q))_{e|m} \mapsto (q^m-1)(c_{e/d}(q^d))_{d|e|m}$. In terms of Habiro rings, if $R$ is étale over~$\Z$, then this morphism is given by the composite
    $$\frac{(q^m-1)\mathcal{H}_{R,m}}{(q^m-1)^2\mathcal{H}_{R,m}} \xlongrightarrow{\Frob_d^{\Hab}} \frac{(q^{dm}-1)\mathcal{H}^{(d)}_{R,dm}}{(q^{dm}-1)^2\mathcal{H}^{(d)}_{R,dm}} \xlongrightarrow{\frac{1}{d}} \frac{(q^{m}-1)\mathcal{H}^{(d)}_{R,m}}{(q^{m}-1)^2\mathcal{H}^{(d)}_{R,m}}$$
    where the first morphism is induced by the Habiro Frobenius of Construction~\ref{constructionHabiroFrobenii}. By construction, the truncated $d^{th}$ twisted cyclotomic Frobenii are compatible between different integers $m \ge 1$.
\end{construction}

\begin{definition}[Twisted cyclotomic Frobenii]\label{definitiontwistedcyclotomicFrobenius}
    Let $K$ be a number field, $R$ be the étale $\Z$-algebra $\mathcal{O}_K[\Delta_K^{-1}]$, and $d \ge 1$ be an integer. The \emph{$d^{th}$ twisted cyclotomic Frobenius} is the morphism of abelian groups
    $$\Frob_d^{\Cyc}\{1\} : \mathcal{N}^{\ge 1} \CR\{1\} \longrightarrow \CR^{(d)}\{1\}$$
    defined as the inverse limit over integers $m \ge 1$ satisfying $(m,N)=1$ of the $m$-truncated twisted cyclotomic Frobenii of Construction~\ref{constructiontruncatedtwistedcyclotomicFrobenius}. 
\end{definition}

\begin{definition}[Cyclosyntomic cohomology]\label{definitioncyclosyntomiccohomology} 
    Let $K$ be a number field, $R$ be the étale $\Z$-algebra $\mathcal{O}_K[\Delta_K^{-1}]$, $N$ be a multiple of $\Delta_K$, and $d \ge 1$ be an integer. The \emph{cyclosyntomic complex of $R$ at $d$} is the object
    $$R\Gamma_{\CycSyn}(R,\Z(1)^{(d)}) := \left[ \mathcal{N}^{\ge 1} \mathcal{C}_R\{1\} \xlongrightarrow{\can - \Frob^{\Cyc}_d\{1\}} \mathcal{C}_R^{(d)}\{1\} \right]$$
    in the derived category $\mathcal{D}(\Z)$, where the first term $\mathcal{N}^{\ge 1} \CR\{1\}$ (Definition~\ref{definitionNygaardtwistedcyclotomicring}) sits in cohomological degree zero, the second term $\CR^{(d)}\{1\}$ (Definition~\ref{definitiontwistedcyclotomicringatd}) sits in cohomological degree one, and the morphism is defined by Notation~\ref{notationcanonicalmapcyclosyntomic} and Definition~\ref{definitiontwistedcyclotomicFrobenius}.
\end{definition}

\begin{remark}
    Although variants\footnote{For instance, if one does not restrict the limit over $m \ge 1$ to integers satisfying $(m,N)=1$ in Definitions~\ref{definitionNygaardtwistedcyclotomicring} and~\ref{definitiontwistedcyclotomicringatd}.} of the previous definition would make sense for arbitrary commutative rings~$R$, we do not expect it to be a reasonable definition beyond the case of étale $\Z$-algebras $R$. This intuition is directly drawn from the definition of syntomic cohomology in terms of absolute prismatic cohomology \cite{bhatt_prisms_2022,bhatt_absolute_2022}.
\end{remark}

\begin{remark}[Comparison to syntomic cohomology]\label{remarkcomparisontosyntomiccohomology}
    Let $K$ be a number field, $R$ be the étale $\Z$\nobreakdash-algebra $\nbrrg$, and $p$ be a prime number. Here we compare the cyclosyntomic complex of $R$ of Definition~\ref{definitioncyclosyntomiccohomology} with the prismatic approach to syntomic cohomology, as developed in \cite{bhatt_topological_2019,bhatt_prisms_2022,bhatt_absolute_2022,antieau_prismatic_2023}. Following \cite[Section~7]{antieau_prismatic_2023}, the (weight one) syntomic complex of $R$ relative to the $q$-prism $(\Z_p[\![q-1]\!],[p]_q)$ is the object
    $$R\Gamma_{q\text{syn}}(R,\Z_p(1)) := \fib\Big(\mathcal{N}^{\ge 1} \Prism_{R^{(1)}/\Z_p[\![q-1]\!]}\{1\} \xlongrightarrow{\can - \Frob_p^{\Prism}\{1\}} \Prism_{R^{(1)}/\Z_p[\![q-1]\!]}\{1\}\Big)$$
    of the derived category $\mathcal{D}(\Z_p)$, where $R^{(1)} := R \otimes_{\Z_p} \Z_p[\zeta_p]$ and $\Frob_p^{\Prism}\{1\}$ is the divided Frobenius on the first Breuil--Kisin twist of prismatic cohomology. Given that $R=\nbrrg$ is étale over~$\Z$, the prismatic cohomology of $R^{(1)}$ relative to $\Z_p[\![q-1]\!]$ is concentrated in degree zero, where it is given by the initial prism $R^\wedge_p[\![q-1]\!]$ of the prismatic site $(R^{(1)}/\Z_p[\![q-1]\!])_{\Prism}$. Similarly, the Nygaard twist $\mathcal{N}^{\ge 1} \Prism_{R^{(1)}/\Z_p[\![q-1]\!]}\{1\}$ is given by $(q-1)R^{\wedge}_p[\![q-1]\!]$ concentrated in degree zero (\cite[Section~15]{bhatt_prisms_2022}, see also \cite[3.20]{wagner_q-hodge_2025}). 
    Unwinding the definitions, the $p$-adic realisation maps of Remarks~\ref{remarkpadicrealisationcyclotomic} and~\ref{remarkpadicrealisationHabiroandcyclotomicFrobenius} then induce a natural commutative diagram
    $$\begin{tikzcd}[column sep = huge]
        \mathcal{N}^{\ge 1} \CR\{1\} \ar[r,"\can - \Frob_d^{\Cyc}\{1\}"] \ar[d] & \CR^{(p)}\{1\} \ar[d] \\
        (q-1)R^\wedge_p[\![q-1]\!] \ar[r,"\can-\Frob_p^{\Prism}\{1\}"] & R^\wedge_p[\![q-1]\!]
    \end{tikzcd}$$
    in the derived category $\mathcal{D}(\Z)$, where all the terms sit in cohomological degree zero. In particular, this commutative diagram induces a natural comparison map
    $$R\Gamma_{\CycSyn}(R,\Z(1)^{(p)}) \longrightarrow R\Gamma_{q\text{syn}}(R,\Z_p(1))$$
    in the derived category $\mathcal{D}(\Z)$.
\end{remark}

\begin{remark}[Truncated cyclosyntomic cohomology]\label{remarktruncatedcyclosyntomiccohomology}
    Given a number field $K$, $R:=\mathcal{O}_K[\Delta_K^{-1}]$, a multiple $N$ of $\Delta_K$, and an integer $d \ge 1$, the cyclosyntomic complex $R\Gamma_{\CycSyn}(R,\Z(1)^{(d)})$ of Definition~\ref{definitioncyclosyntomiccohomology} is the limit over integers $m \ge 1$ satisfying $(m,N)=1$ of the analogous complexes
    $$R\Gamma_{\CycSyn}(R,\Z(1)^{(d)}_m) := \left[ \mathcal{N}^{\ge 1} \mathcal{C}_{R,m}\{1\} \xlongrightarrow{\can - \Frob^{\Cyc}_d\{1\}} \mathcal{C}_{R,m}^{(d)}\{1\} \right]$$
    in the derived category $\mathcal{D}(\Z)$, where the morphism is given by $$\can - \Frob_d^{\Cyc}\{1\} : (\Phi_e(q)c_e(q))_{e|m} \mapsto (\Phi_e(q)c_e(q) - \Phi_e(q)c_{d/e}(q^d))_{d|e|m}.$$ This is indeed a consequence of Remarks~\ref{remarklimitovermNygaardandtwistedcyclotomic}, Notation~\ref{notationcanonicalmapcyclosyntomic}, and Construction~\ref{constructiontruncatedtwistedcyclotomicFrobenius}.
\end{remark}

\subsection{The first Chern class}\label{subsec:the-first-chern-class}

In this subsection, we define the cyclosyntomic first Chern class (Definition~\ref{definitioncyclosyntomicfirstChernclass}) as a refinement of the cyclotomic logarithm $\dlog_{\Cyc} :\mathbb{G}_m(R)[-1]\to \mathcal{N}^{\ge 1} \mathcal{C}_{R}\{1\}$
of Section~\ref{subsectioncyclotomiclogarithm}. More precisely, we prove that the cyclotomic logarithm factors through the cyclosyntomic complex defined in Section~\ref{subsectionthecyclosyntomiccohomologyofanumberfield}, by exhibiting suitable null-homotopies at the truncated level (Construction~\ref{constructionhomotopyforfirstChernclass} and Proposition~\ref{propositionfirstChernclasszero}).

\begin{notation}\label{notationextensionclassEm}
    Given an étale $\Z$-algebra $R$ and an integer $m \ge 1$, we denote by
    $$E_m := \big\{(y,z) \in \G_m(\mathcal{H}_{R,m}/(q^m-1)^2) \times \G_m(R)~\big \vert~\overline{y} = \Pi_m(z) \text{ in } \mathcal{H}_{R,m}/(q^m-1) \cong \qWm(R)\big\}$$
    the abelian group that corresponds to the extension class
    $$0 \longrightarrow \mathcal{N}^{\ge 1} \CRm\{1\} \xlongrightarrow{\exp \times 1} E_m \longrightarrow \G_m(R) \longrightarrow 0$$
    given by the $m$-truncated cyclotomic logarithm $\dlog_{\Cyc}^{(m)}$ of Construction~\ref{constructiontruncatedcyclotomiclogarithm}. Here we use the isomorphism
    $$\mathcal{N}^{\ge 1} \CRm\{1\}\cong (q^m-1)\mathcal{H}_{R,m}/(q^m-1)^2\mathcal{H}_{R,m}$$ of Remark~\ref{remarkNygaardwithHabiro} to make sense of this short exact sequence.
\end{notation}

The following construction will be used to define the aforementioned null-homotopy (Proposition~\ref{propositionfirstChernclasszero}) and to give the formula presented in the introduction for the induced cyclosyntomic first Chern class (Corollary~\ref{corollaryformulaforfirstChernclass}).

\begin{construction}\label{constructionhomotopyforfirstChernclass}
    Let $R$ be an étale $\Z$-algebra, and $d \ge 2$ and $m \ge 1$ be integers. We construct a morphism of abelian groups
    $$s_d : E_m \longrightarrow \CRm^{(d)}\{1\}$$
    where $E_m$ is defined in Notation~\ref{notationextensionclassEm}. 
    This morphism is given by the expression
    $$s_d : (y,z) \longmapsto \frac{1}{d} \log\Big(\frac{\tilde{\Pi}_d(y)}{\Frob^{\Hab}_d(y)}\Big)$$
    where the formula is understood via the following conventions: 
    \begin{itemize}
        \item $\tilde{\Pi}_d := \can\circ \tilde{\Pi}_d$ is the composite map 
         $$ \G_m\Big(\frac{\mathcal{H}_{R,m}}{(q^m-1)^2}\Big) \xlongrightarrow{\tilde{\Pi}_d } \G_m\Big(\frac{\mathcal{H}_{R,dm}}{(q^m-1)(q^{dm}-1)}\Big) \xlongrightarrow{\can} \G_m\Big(\frac{\mathcal{H}_{R,dm}^{(d)}}{(q^m-1)(q^{dm}-1)}\Big)$$
         where the first map $\tilde{\Pi}_d$ is the lifted cyclotomic norm of Proposition~\ref{propositionliftedcyclotomicpower}; 
         \item $\Frob_d^{\Hab}$ is the composite map
         $$\G_m\Big(\frac{\mathcal{H}_{R,m}}{(q^m-1)^2}\Big) \xlongrightarrow{\Frob_d^{\Hab}} \G_m\Big(\frac{\mathcal{H}_{R,dm}^{(d)}}{(q^{dm}-1)^2}\Big) \relbar\joinrel\twoheadrightarrow \G_m\Big(\frac{\mathcal{H}_{R,dm}^{(d)}}{(q^m-1)(q^{dm}-1)}\Big)$$
         where the first map $\Frob_d^{\Hab}$ is the Habiro Frobenius of Construction~\ref{constructionHabiroFrobenii}; 
    \item the logarithm is the morphism of abelian groups
    $$\log : \G_m\Big(\frac{\mathcal{H}_{R,dm}^{(d)}}{(q^m-1)(q^{dm}-1)}\Big) \longrightarrow \frac{\mathcal{H}_{R,dm}^{(d)}}{(q^m-1)(q^{dm}-1)}, \quad t\longmapsto t-1;$$
    \item ``division by $d$'' is understood as the morphism of abelian groups
    $$d^{-1} =[d]_{q^m}^{-1} : \frac{[dm]_q \mathcal{H}_{R,dm}^{(d)}}{[m]_q[dm]_q\mathcal{H}_{R,dm}^{(d)}} \xlongrightarrow{\cong} \frac{[m]_q \mathcal{H}_{R,dm}^{(d)}}{[m]_q^2\mathcal{H}_{R,dm}^{(d)}} \cong \CRm^{(d)}\{1\}$$
    given by division by $[d]_{q^m} = [dm]_q/[m]_q$ (Lemma~\ref{lemmadivisionbym/e}\,(3)).
    \end{itemize}
    To check that the morphism $s_d$ is well-defined, it suffices to prove, for every $(y,z)\in E_m$, that the element
    $$\frac{\tilde{\Pi}_d(y)}{\Frob_d^{\Hab}(y)}-1\in \mathcal{H}_{R,dm}^{(d)}/[m]_q[dm]_q \mathcal{H}_{R,dm}^{(d)}$$ belongs to the ideal $([dm]_q)$, or equivalently that $\tilde{\Pi}_d(y) \equiv \Frob_d^{\Hab}(y)$ in $\mathcal{H}_{R,dm}^{(d)}/[dm]_q$. We now use Lemma~\ref{lemmaqghostmapinjectiveHabiro}, whose proof adapts readily from the Habiro ring $\mathcal{H}_{R,m}$ (Variant~\ref{varianttruncatedHabiroring}) to the variant $\mathcal{H}_{R,m}^{(d)}$ (Variant~\ref{variantHabiroringatd}), in order to reduce this statement to a statement on $q$-ghost coordinates.  
    More precisely, it then suffices to prove this identity on the $e^{th}$ $q$-ghost coordinate for integers $e \ge 1$ satisfying $d|e|dm$.\footnote{Note that this is the critical step where we need $\Phi_e(q)$ to be invertible in $\mathcal{H}_{R,dm}^{(d)}$ for $d \nmid e$, and the reason for taking this convention in Variant~\ref{variantcyclotomicringatd}, Variant~\ref{variantHabiroringatd}, and Definition~\ref{definitiontwistedcyclotomicringatd}.} 
    For every such integer $e$, the desired identity follows from the series of equalities in~$R[q]/\Phi_e(q)$:
    $$\tilde{\Pi}_d(y)_e = \Pi_d(y)_e= \Pi_d(\Pi_m(z))_e = \Pi_{dm}(z)_e = z^{\frac{dm}{e}} = z^{\frac{m}{e/d}} = \Frob_d^{\Hab}(\Pi_m(z))_e = \Frob_d^{\Hab}(y)_e$$
    where we use that $y\equiv \Pi_m(z)\pmod{\Phi_e(q)}$ by definition of $(y,z) \in E_m$ (Notation~\ref{notationextensionclassEm}).
\end{construction}

\begin{proposition}\label{propositionfirstChernclasszero}
    Let $R$ be an étale $\Z$-algebra, and $d \geq 2$ be an integer. Then for every integer $m \ge 1$, the map $$s_d : E_m \longrightarrow \CRm^{(d)}\{1\}$$ 
    of Construction~\ref{constructionhomotopyforfirstChernclass} induces a natural homotopy $h'_m$ making the composite
    $$\G_m(R)[-1] \xlongrightarrow{\dlog_\Cyc^{(m)}} \mathcal{N}^{\ge 1} \mathcal{C}_{R,m}\{1\} \xlongrightarrow{\emph{can} - \Frob_d^{\Cyc}\{1\}} \mathcal{C}^{(d)}_{R,m}\{1\}$$
    homotopic to zero in the derived category $\mathcal{D}(\Z)$.
\end{proposition}

\begin{proof}
    This composite map corresponds to an extension in $\text{Ext}^1_{\Z}(\G_m(R),\CRm^{(d)}\{1\})$, given by the pushout of the extension $[\dlog_{\Cyc}^{(m)}]\in  \text{Ext}^1_{\Z}(\G_m(R),\mathcal{N}^{\ge 1} \mathcal{C}_{R,m}\{1\})$ of Construction~\ref{constructiontruncatedcyclotomiclogarithm} along the map $\text{can} - \Frob_d^{\Cyc}\{1\}$. 
    Unwinding the definition of the map $\dlog_{\Cyc}^{(m)}$, a splitting of this extension is in turn equivalent to a morphism of abelian groups
    $$s_d : E_m \longrightarrow \CRm^{(d)}\{1\}$$
    such that the diagram of abelian groups
    $$\begin{tikzcd}[column sep = large]
        \mathcal{N}^{\ge 1} \CRm\{1\} \ar[d,"\can - \Frob_d^{\Cyc}\{1\}"'] \ar[r,"\exp \times 1"] & E_m \ar[ld,"s_d"] \\
        \CRm^{(d)}\{1\} &
    \end{tikzcd}$$
    is commutative.  
    We claim that the morphism $s_d$ of Construction~\ref{constructionhomotopyforfirstChernclass} satisfies this property. To prove this, let $x$ be an element of $\mathcal{N}^{\ge 1}\CRm\{1\}$. We then have the series of equality
    \begin{align*}
    \frac{1}{d}\log\Big(\frac{\tilde{\Pi}_d(\exp(x))}{\Frob_d^{\Hab}(\exp(x))}\Big) &= \frac{1}{d} \Big(\frac{1+d\,\can(x)}{1+\Frob_d^{\Hab}(x)} - 1\Big) \\
    &= \frac{1}{d} \big((1+d\,\can(x))(1-\Frob_d^{\Hab}(x)) - 1 \big) \\
    &= \frac{1}{d}\big(d\,\can(x) - \Frob_d^{\Hab}(x)\big) \\
    &= \can(x) - \Frob_d^{\Cyc}\{1\}(x)
    \end{align*}
    where the first equality is a consequence of Corollary~\ref{corollaryliftedcyclotomicpowercommutativediagram}, the second and third equalities are consequences of the fact that $d\,\can(x)$ and $\Frob_d^{\Hab}(x)$ are divisible by $(q^{dm}-1)$, and the last equality is a consequence of Construction~\ref{constructiontruncatedtwistedcyclotomicFrobenius}. 
\end{proof}

\begin{construction}[Truncated cyclosyntomic first Chern class]\label{constructiontruncatedcyclosyntomicfirstChernclass}
    Let $R$ be an étale $\Z$-algebra, and $d \ge 2$ and $m \ge 1$ be integers. By Proposition~\ref{propositionfirstChernclasszero}, the cyclotomic logarithm 
    $$\dlog^{(m)}_{\Cyc} : \G_m(R)[-1] \longrightarrow \mathcal{N}^{\ge 1} \mathcal{C}_{R,m}\{1\}$$
    factors uniquely, in the derived category $\mathcal{D}(\Z)$, through the homotopy fibre of the map
    $$\can - \Frob_d^{\Cyc}\{1\} : \mathcal{N}^{\ge 1} \mathcal{C}_{R,m}\{1\} \longrightarrow \CRm^{(d)}\{1\}.$$
    We denote by
    $$c_1^{\CycSyn} : \G_m(R)[-1] \longrightarrow R\Gamma_{\CycSyn}(R,\Z(1)^{(d)}_m)$$
    the induced map in the derived category $\mathcal{D}(\Z)$, where the target is defined in Remark~\ref{remarktruncatedcyclosyntomiccohomology}.
\end{construction}

\begin{remark}\label{remarklimitovermfirstChernclass}
    The homotopies $h'_m$ of Proposition~\ref{propositionfirstChernclasszero} are compatible between different integers $m \ge 1$. To prove this, it indeed suffices to prove that the diagram of abelian groups
    $$\begin{tikzcd}[column sep = large]
        E_{m'} \arrow[r,"h_{m,m'}"] \arrow[d,"s_d"'] & E_m \arrow[d,"s_d"] \\
    \mathcal{C}^{(d)}_{R,m'}\{1\} \arrow[r,"F_{m'/m}\{1\}"] & \mathcal{C}^{(d)}_{R,m}\{1\}
    \end{tikzcd}$$
    is commutative for all integers $m,m' \ge 1$ satisfying $m|m'$, and this is a consequence of the compatibility between Construction~\ref{constructionhomotopyforfirstChernclass} and the homotopies $h_{m,m'}$ of Proposition~\ref{propositioncompatiblehomotopies}.
\end{remark}

\begin{remark}[\'Etale descent]\label{remarketaledescent}
    By \cite[Proposition~2.15]{wagner_q-witt_2024}, any descent property satisfied by the identity functor on commutative rings is also satisfied by the functor $R \mapsto \W_m(R)$ for every integer $m \ge 1$. Unwinding the definitions, this in particular implies that the cyclosyntomic complex of étale $\Z$\nobreakdash-algebras $R$ (Definition~\ref{definitioncyclosyntomiccohomology} and Remark~\ref{remarktruncatedcyclosyntomiccohomology}) satisfies étale descent on $R$, and the cyclosyntomic first Chern class of Construction~\ref{constructiontruncatedcyclosyntomicfirstChernclass} then naturally factors as a map
    $$c_1^{\CycSyn} : R\Gamma_{\text{ét}}(R,\G_m)[-1] \longrightarrow R\Gamma_{\CycSyn}(R,\Z(1)^{(d)}_m)$$
    in the derived category $\mathcal{D}(\Z)$, for any integers $d \ge 2$ and $m \ge 1$. Note here that we use the fact that the higher coherent cohomology of affine schemes vanishes.
\end{remark}

\begin{definition}[Cyclosyntomic first Chern class]\label{definitioncyclosyntomicfirstChernclass}
    Let $K$ be a number field, $R$ be the étale $\Z$\nobreakdash-algebra $\mathcal{O}_K[\Delta_K^{-1}]$, $N$ be a multiple of $\Delta_K$, and $d \ge 2$ be an integer. The \emph{cyclosyntomic first Chern class of $R$ at $d$} is the map
    $$c_1^{\CycSyn} : \G_m(R)[-1] \longrightarrow R\Gamma_{\CycSyn}(R,\Z(1)^{(d)})$$
    in the derived category $\mathcal{D}(\Z)$, defined as the inverse limit over integers $m \ge 1$ satisfying $(m,N)=1$ of the $m$-truncated maps of Construction~\ref{constructiontruncatedcyclosyntomicfirstChernclass}. Note here that we use Remark~\ref{remarklimitovermfirstChernclass} to make sense of this inverse limit.
\end{definition}

The following result is a global analogue of the fundamental $p$-adic computation of Kato (\cite[Corollary~2.9]{kato_explicit_1991}, see also \cite[Proposition~4.1 and Appendix]{gros_regulateurs_1990} or \cite[page 289]{somekawa_log-syntomic_1999}).

\begin{corollary}\label{corollaryformulaforfirstChernclass}
    Let $K$ be a number field, $R$ be the étale $\Z$-algebra $\mathcal{O}_K[\Delta_K^{-1}]$, and $d \ge 2$ be an integer. For every unit $u \in \G_m(R)$, the cyclosyntomic first Chern class at $u$ is given by
    $$c_1^{\CycSyn} : u \longmapsto \Big(\frac{1}{d}\log\Big(\frac{\tilde{\Pi}_d(\Pi_m^{\Hab}(u))}{\Frob_d^{\Hab}(\Pi_m^{\Hab}(u))}\Big)\Big)_{m \ge 1} \in \mathrm{H}^1_{\CycSyn}(R,\Z(1)^{(d)})$$
    where $\Pi_m^{\Hab}(u) \in \mathcal{H}_{R,m}/(q^m-1)^2$ is any lift of the element $\Pi_m(u) \in \qWm(R)$ via the surjective map of commutative $\Z[q]$-algebras $\mathcal{H}_{R,m}/(q^m-1)^2 \twoheadrightarrow \qWm(R)$ (Remark~\ref{remarkHabiroqWitt}), $\tilde{\Pi}_d$ is the lifted cyclotomic norm (Proposition~\ref{propositionliftedcyclotomicpower}), and $\Frob_d^{\Hab}$ is the Habiro Frobenius (Construction~\ref{constructionHabiroFrobenii}).
\end{corollary}

\begin{proof}
    By Definition~\ref{definitioncyclosyntomicfirstChernclass}, it suffices to prove the result at the truncated level, where this is a consequence of Constructions~\ref{constructionhomotopyforfirstChernclass} and~\ref{constructiontruncatedcyclosyntomicfirstChernclass}.
\end{proof}

\subsection{The first $q$-polylogarithm}\label{subsectionthefirstqpolylogarithm}

In this subsection, we compute the cyclosyntomic first Chern class of the previous subsection at cyclotomic units in terms of a $q$-deformation of the first polylogarithm (Theorem~\ref{theoremmainLi1cyclosyntomic}). To do so, we first introduce the relevant $q$-deformation of the first polylogarithm, whose definition is motivated by Deligne's notion of $p$-adic polylogarithm (\cite[3.2.3]{deligne_groupe_1989}).

\begin{construction}\label{constructionfirstpolylogarithmformal}
    Let $d \ge 2$ be an integer, and consider the power series
    $$\Li_1^{(d)}(T)_q := \sum_{\substack{k \ge 1 \\ d \nmid k}} \frac{T^k}{[k]_q} \in \Z[q]\left[\frac{1}{[k]_q}~\bigg|~\substack{k \ge 1 \\ d \nmid k}\right][\![T]\!].$$
    Given any integer $m \ge 1$, as a consequence of the identity $[k+m]_q = [k]_q + q^k[m]_q$, the reduction modulo $[m]_q$ of this powers series is
    $$\Li_1^{(d)}(T)_q \equiv \sum_{a \ge 0} \sum_{\substack{0 < k < m \\ d \nmid k}} \frac{T^{k+am}}{[k]_q} = \Big(\sum_{a \ge 0} T^{am}\Big) \Big( \sum_{\substack{0 < k < m \\ d \nmid k}} \frac{T^k}{[k]_q} \Big) = \frac{1}{1-T^m} \sum_{\substack{0 < k < m \\ d \nmid k}} \frac{T^k}{[k]_q}.$$
    In particular, the $m$-truncated first $q$-polylogarithms
    $$\Li_1^{(d)}(T)_q := \frac{1}{1-T^m} \sum_{\substack{0 < k < m \\ d \nmid k}} \frac{T^k}{[k]_q} \in \mathbb{Z}[q,T]\left[\frac{1}{1-T^m},\frac{1}{[k]_q}~\bigg|~\substack{0<k<m \\ d\nmid k}\right]/[m]_q$$
    are compatible between different integers $m \ge 1$.
\end{construction}

\begin{notation}\label{notationclassofzeta}
    Let $K$ be a number field, $R$ be the étale $\Z$-algebra $\mathcal{O}_K[\Delta_K^{-1}]$, 
    and $\zeta \in R$ be a root of unity. For every integer $m \ge 1$ which is coprime to $\Delta_K$, we denote by
    $$[\zeta] := (\zeta^{\frac{1}{e}})_{e|m} \in \mathcal{H}_{R,m}$$
    the class associated to $\zeta$ in the Habiro ring $\mathcal{H}_{R,m}$ (Variant~\ref{varianttruncatedHabiroring}),
    and similarly for the induced element in any $\mathcal{H}_{R,m}$-algebra (\emph{e.g.}, $[\zeta] \in \mathcal{H}_{R,m}^{(d)}$ for any integer $d \ge 1$). Note that this class $[\zeta]$ is well-defined in $\mathcal{H}_{R,m}$ because for any integer $m \ge 1$ which is coprime to $\Delta_K$, $\zeta$ admits an $m^{th}$ root $\zeta^{\frac{1}{m}}$ in $R$, and for any integer $e \ge 1$ and prime number $p$ such that $pe|m$, the Frobenius map $\varphi_{R^\wedge_p} : R^\wedge_p \rightarrow R^\wedge_p$ sends $\zeta^{\frac{1}{pe}}$ to $\zeta^{\frac{1}{e}}$.
\end{notation}

\begin{definition}[First $q$-polylogarithm class]\label{definitionpolylogarithmcyclosyntomic}
    Let $K$ be a number field, $R$ be the étale $\Z$-algebra $\mathcal{O}_K[\Delta_K^{-1}]$, $d \ge 2$ and $m \ge 1$ be integers, and $\zeta \in R \setminus\{\pm 1\}$ be a root of unity that admits an $m^{th}$ root $\zeta^{\frac{1}{m}}$ in $R$. The \emph{$m$-truncated first $q$-polylogarithm class of $\zeta$} is the element
    $$\Li_1^{(d)}([\zeta])_q := \Big( \frac{1}{1-[\zeta]^m} \sum_{\substack{0 < k < m \\ d \nmid k}} \frac{[\zeta]^k}{[k]_q}\Big) \cdot [m]_q \in \CRm^{(d)}\{1\}$$
    where $\CRm^{(d)}\{1\}$ is defined in Definition~\ref{definitiontwistedcyclotomicringatd}, $[\zeta] \in \mathcal{H}_{R,m}^{(d)}$ is defined in Notation~\ref{notationclassofzeta}, 
    and where we use the isomorphism of $\Z[q]$-modules $\CRm^{(d)}\{1\} \cong [m]_q\mathcal{H}_{R,m}^{(d)}/[m]_q^2\mathcal{H}_{R,m}^{(d)}$ of Remark~\ref{remarklimitovermNygaardandtwistedcyclotomic}.\footnote{Note here that to evaluate the first $q$-polylogarithm of Construction~\ref{constructionfirstpolylogarithmformal} in the $\Z[q]$-module $\mathcal{C}_{R,m}^{(d)}\{1\}$, we use that $[k]_q = \prod_{e|k,e>1} \Phi_e(q)$ is invertible (because each $\Phi_e(q)$ is, if $d \nmid k$) and that $1-[\zeta]^m$ is invertible (because the order of $\zeta$ is invertible in $R$, and one only considers the $e^{th}$ $q$-ghost coordinates for $(e,N) = 1$, so that $1-\zeta^{\frac{1}{e}}$ is invertible in $R$).} These elements are compatible between different integers $m \ge 1$ by Construction~\ref{constructionfirstpolylogarithmformal}, and we call \emph{first $q$-polylogarithm class of $\zeta$} and denote by
    $$\Li_1^{(d)}([\zeta])_q \in \CR^{(d)}\{1\}$$
    the induced compatible system of elements. 
    Similarly, we denote by 
    $$\Li_1^{(d)}([\zeta])_q \in \mathrm{H}^1_{\CycSyn}(R,\Z(1)^{(d)})$$
    the element induced via the canonical map of abelian groups $\CR^{(d)}\{1\} \twoheadrightarrow \text{H}^1_{\CycSyn}(R,\Z(1)^{(d)})$.
\end{definition}

\begin{remark}[Comparison with \cite{gros_absolute_2023}]
    Let $k$ be a perfect field of characteristic $p$, and $R$ be the étale $\Z_p$-algebra $W(k)$. Following the work of Gros--Le Stum--Quirós \cite{gros_absolute_2023}, the absolute prismatic cohomology of $R^{(1)}:=R \otimes_{\Z_p} \Z_p[\zeta_p]$ can be computed by the complex
    $$\Prism_{R^{(1)}} \simeq \left[ R[\![q-1]\!]  \xlongrightarrow{\partial_{\Prism}} R[\![q-1]\!] \right]$$
    where the prismatic derivation $\partial_{\Prism}$ is given by $\partial_{\Prism}(q^k):=[pk]_q q^{k-1}$. This identification is compatible with Breuil--Kisin twists and the Nygaard filtration, thus inducing a similar computation of the absolute syntomic cohomology of $R^{(1)}$. Moreover, the first $q$-polylogarithm classes of Definition~\ref{definitionpolylogarithmcyclosyntomic} induce, for $d=p$ and via the $p$-adic comparison map of Remark~\ref{remarkcomparisontosyntomiccohomology}, natural $q$\nobreakdash-polylogarithm classes in the relative syntomic cohomology group $\text{H}^1_{q\text{syn}}(R,\Z_p(1))$. We expect that these classes should come from natural cohomology classes in the absolute syntomic cohomology group $\text{H}^1_{\text{syn}}(R^{(1)},\Z_p(1))$, via the canonical map $\text{H}^1_{\text{syn}}(R^{(1)},\Z_p(1)) \rightarrow \text{H}^1_{q\text{syn}}(R,\Z_p(1))$. However, we have not been able, using the previous prismatic derivation $\partial_{\Prism}$, to exhibit such cohomology classes. 
\end{remark}

\begin{lemma}\label{lemmaliftof1-zeta}
    Let $K$ be a number field, $R$ be the étale $\Z$-algebra $\mathcal{O}_K[\Delta_K^{-1}]$, $\zeta \in R$ be a root of unity of order dividing $N$. For every integer $m \ge 1$ satisfying $(m,N)=1$, the element
    $$\Pi_m^{\Hab}(1-[\zeta]) := \prod_{0 \le j < m}(1-q^j[\zeta]) \in \mathcal{H}_{R,m}$$
    is sent to the element $\Pi_m(1-\zeta) \in \qWm(R)$ via the canonical surjective map of commutative $\Z[q]$-algebras $\mathcal{H}_{R,m} \twoheadrightarrow \mathcal{H}_{R,m}/(q^m-1) \cong \qWm(R)$ (Remark~\ref{remarkHabiroqWitt}).
\end{lemma}

\begin{proof}
    Let $e \ge 1$ be an integer satisfying $e|m$, and in particular $(e ,N) = 1$. By \cite[Lemma~2.23]{wagner_q-witt_2024}, it suffices to prove that $\Pi_m^{\Hab}(1-[\zeta]) \equiv \Pi_m(1-\zeta)$ modulo $\Phi_e(q)$, which follows from the series of equalities
    $$\prod_{0 \le j < m}(1-q^j\zeta^{\frac{1}{e}}) = \prod_{0 \le a < \frac{m}{e}} \prod_{0 \le j < e}(1-q^{ae+j}\zeta^{\frac{1}{e}}) \equiv \prod_{0 \le j < e} (1-q^j \zeta^{\frac{1}{e}})^{\frac{m}{e}} \equiv (1-\zeta)^{\frac{m}{e}}$$
    where we use that $q$ is a primitive $e^{th}$ root of unity modulo $\Phi_e(q)$.
\end{proof}

\begin{proposition}\label{propositionkeycomputation}
    For any integers $d \ge 2$ and $m \ge 1$, the equality
    $$\frac{1}{d}\log\left(\frac{\prod_{0 \le j < m}{(1-q^j T)}^d}{\prod_{0 \le j < m}{(1-q^{jd} T^d)}}\right)\equiv -[m]_q \cdot \frac{1}{1-T^m}\sum_{\substack{0 < k < m \\ d \nmid k}} \frac{T^k}{[k]_q}$$
    holds in the commutative ring $\mathbb{Z}[q,T]\left[\frac{1}{1-T^m},\frac{1}{[k]_q},\frac{1}{1-q^jT}~\bigg|~\substack{0<k<m \\ d\nmid k},0 \le j < m\right]/[m]_q^2$.
\end{proposition}

\begin{proof}
    We first prove that
    $$[m]_q^2\Z[q,T][S_1^{-1}] = [m]_q^2\Q[q][S_2^{-1}][\![T]\!] \cap \Z[q,T]$$
    where $S_1 \subseteq \Z[q,T]$ is the multiplicative set generated by $1-T^m$, $[k]_q$ for integers $0<k<m$ satisfying $d \nmid k$, and $1-q^jT$ for $0 \le j < m$, and where $S_2 \subseteq \Q[q]$ is the multiplicative set generated by $[k]_q$ for integers $0<k<m$ satisfying $d \nmid k$. The left hand side is a subset of the right hand side by construction. The other inclusion is a consequence of the fact that the intersection $[m]_q^2\Q[q][S_2^{-1}] \cap \Z[q]$ is equal to $\prod_{d|e|m} \Phi_e(q)^2 \Z[q]$. 
    
    Using the previous paragraph, it suffices to prove the desired equality in the bigger commutative $\Q[q][S_2^{-1}][\![T]\!]/[m]_q^2$, where the computation
    \[
    \log\left(\prod_{0 \le j < m}{(1-q^j T)}\right)=\sum_{0 \le j < m}{\log(1-q^j T)}=-\sum_{0 \le j < m}{\sum_{k\geq 1}{\frac{(q^j T)^k}{k}}}=-\sum_{k\geq 1}{\left(\frac{q^{mk}-1}{q^k-1}\right) \frac{T^k}{k}}
    \]
    makes sense. In particular, this implies that
    \begin{align*}
    \frac{1}{d}\log\left(\frac{\prod_{0 \le j < m}{(1-q^j T)}^d}{\prod_{0 \le j < m}{(1-q^{jd} T^d)}}\right) &=-\sum_{\substack{k \geq 1 \\ d\nmid k}}{\left(\frac{q^{mk}-1}{q^k-1}\right) \frac{T^k}{k}} \\
    &=-[m]_q\sum_{\substack{k\geq 1 \\ d\nmid k}}{\left(\frac{[k]_{q^m}}{[k]_q}\right) \frac{T^k}{i}}\\
    &\equiv-[m]_q\cdot \sum_{\substack{i>0 \\ d\nmid k}}{\frac{T^k}{[k]_q}}
    \end{align*}
    where we use that $[m]_q \cdot [k]_{q^m} = [m]_q \cdot k$ modulo $[m]_q^2$. The desired result is then a consequence of Construction~\ref{constructionfirstpolylogarithmformal}.
\end{proof}

\begin{theorem}\label{theoremmainLi1cyclosyntomic}
    Let $K$ be a number field, $R$ be the étale $\Z$-algebra $\mathcal{O}_K[\Delta_K^{-1}]$, $N$ be a multiple of $\Delta_K$, and $d \ge 2$ be an integer. For every root of unity $\zeta \in R\setminus\{\pm 1\}$,\footnote{For $\zeta = -1$, the same result up to replacing $R$ by $R[\tfrac{1}{2}]$, in order to ensure that $1-\zeta \in \G_m(R)$.} the cyclosyntomic first Chern class (Definition~\ref{definitioncyclosyntomicfirstChernclass}) sends the unit $1-\zeta \in \G_m(R)$ to the first $q$-polylogarithm class $\Li_1^{(d)}([\zeta])_q \in \mathrm{H}^1_{\CycSyn}(R,\Z(1)^{(d)})$ (Definition~\ref{definitionpolylogarithmcyclosyntomic}):
    $$c_1^{\CycSyn} : 1-\zeta \longmapsto -\Li_1^{(d)}([\zeta])_q.$$
\end{theorem}

\begin{proof}
    First note that $1-\zeta$ is a unit in $R$. Indeed, if $g$ is the order of $\zeta$, then $1-\zeta$ divides $\Phi_g(1)$ in $R$, and $\Phi_g(1)$ divides $g$ in $\Z$ ($\Phi_g(1)$ is either equal to $p$ if $g=p^r$ is a prime power, or equal to $1$ otherwise), hence $1-\zeta$ divides $g$ in $R$. Moreover, because $g>2$, we know that the prime supports of $g$ and of $\Delta_{\Q(\zeta)}$ agree. Using that $\Delta_{\Q(\zeta)}$ divides $\Delta_K$ in $\Z$, this implies that $g$ is invertible in $R$, and so is $1-\zeta$.
    
    By Corollary~\ref{corollaryformulaforfirstChernclass} and Definition~\ref{definitionpolylogarithmcyclosyntomic}, it suffices to prove that for every integer $m \ge 1$ which is coprime to $N$ (and in particular to the order of $\zeta$), one has the equality
    $$\frac{1}{d}\log\Big(\frac{\tilde{\Pi}_d(\Pi_m^{\Hab}(1-\zeta))}{\Frob_d^{\Hab}(\Pi_m^{\Hab}(1-\zeta))}\Big) = -\Big( \frac{1}{1-[\zeta]^m} \sum_{\substack{0 < k < m \\ d \nmid k}} \frac{[\zeta]^k}{[k]_q}\Big) \cdot [m]_q$$
    in $\CRm^{(d)}\{1\}$, 
    where $\Pi_m^{\Hab}(1-\zeta) \in \mathcal{H}_{R,m}/(q^m-1)^2$ is any lift of $\Pi_m(1-\zeta) \in \qWm(R)$. By Lemma~\ref{lemmaliftof1-zeta}, one can choose the lift $\Pi_m^{\Hab}(1-\zeta)$ to be given by $\prod_{0 \le j < m} (1-q^j[\zeta])$, where the element $[\zeta] \in \mathcal{H}_{R,m}/(q^m-1)^2$ is defined in Notation~\ref{notationclassofzeta}. 
    
    Arguing as in the proof of Lemma~\ref{lemmaqghostmapinjectiveHabiro}, the commutative ring $\mathcal{H}_{R,dm}^{(d)}/[m]_q[dm]_q$ is flat over $\Z$, and in particular is $d$-torsionfree. This implies that it suffices to do the desired computation after inverting $d$ and thus, following Construction~\ref{constructionhomotopyforfirstChernclass}, to only consider the $e^{th}$ $q$-coordinates for integers $e \ge 1$ satisfying $d|e|m$. 
    By definition of the Habiro Frobenius $\Frob_d^{\Hab}$ (Construction~\ref{constructionHabiroFrobenii}), we then have
    $$\Frob_d^{\Hab}(\Pi_m^{\Hab}(1-\zeta)) = \Frob_d^{\Hab}\Big(\prod_{0 \le j < m} (1-q^j[\zeta])\Big) = \prod_{0 \le j < m} (1-q^{jd}[\zeta]^d).$$
    By definition of the lifted cyclotomic norm $\tilde{\Pi}_d$ (Proposition~\ref{propositionliftedcyclotomicpower}), we similarly have
    $$\tilde{\Pi}_d(\Pi_m^{\Hab}(1-\zeta)) = \tilde{\Pi}_d\Big(\prod_{0 \le j < m} (1-q^j[\zeta])\Big) = \prod_{0 \le j < m} (1-q^j[\zeta])^d$$
    where we use the fact that we restrict to the $e^{th}$ $q$-coordinates for integers $e \ge 1$ satisfying $d|e|m$. 
    In particular, this implies that
    $$\frac{1}{d}\log\Big(\frac{\tilde{\Pi}_d(\Pi_m^{\Hab}(1-\zeta))}{\Frob_d^{\Hab}(\Pi_m^{\Hab}(1-\zeta))}\Big) = \frac{1}{d}\log\Big(\frac{\prod_{0 \le j < m}(1-q^j[\zeta])^d}{\prod_{0 \le j < m}(1-q^{jd}[\zeta]^d)}\Big),$$
    and the desired result is then a consequence of Proposition~\ref{propositionkeycomputation}, which we can apply because the lift $\Pi_d^{\Hab}(1-\zeta) := \prod_{0 \le j < m}(1-q^j[\zeta])$ of $\Pi_m(1-\zeta)$ is a unit in $\mathcal{H}_{R,m}^{(d)}/[m]_q^2$.
\end{proof}

\begin{remark}
    Let $K$ be a number field, $d \ge 2$ be an integer, and $\zeta \in \nbrrg\setminus\{\pm 1\}$ be a root of unity. Using the equality $(1-\zeta) = -\zeta \cdot (1-\zeta^{-1})$, Theorem~\ref{theoremmainLi1cyclosyntomic} implies in particular that
    $$\Li_1^{(d)}([\zeta])_q = \Li_1^{(d)}([\zeta]^{-1})_q$$
    in $\text{H}^1_{\CycSyn}(\nbrrg,\Z(1)^{(d)})$.
    Here we use that the cyclosyntomic first Chern class vanishes at the root of unity $-\zeta \in R$, which is a consequence of Corollary~\ref{corollaryformulaforfirstChernclass}. Interestingly, while the general relation $\Li_1^{(d)}(T)_q = \Li_1^{(d)}(T^{-1})_q$ does not hold as abstract power series, it does hold in the commutative ring introduced in Proposition~\ref{propositionkeycomputation}. In particular, the previous relation at $[\zeta]$ already holds in the cyclotomic twist $\CR^{(d)}\{1\}$.
\end{remark}

\begin{remark}\label{remarkpadicqpolylogarithminsyntomiccohomology}
    Via the $p$-adic realisation map of Remark~\ref{remarkcomparisontosyntomiccohomology}, Theorem~\ref{theoremmainLi1cyclosyntomic} also implies that the first $q$-polylogarithm class $-\Li_1^{(p)}([\zeta])_q$ also appears naturally in the relative syntomic cohomology group $\text{H}^1_{q\text{syn}}(R,\Z_p(1))$, as the syntomic first Chern class of $1-\zeta$. To the best of our knowledge, even this $q$-refinement of the classical $p$-adic result is new, and complements the work of Anschütz and Le Bras \cite{anschutz_p-completed_2020}, where they proved that a $q$-deformation of the $p$-adic \emph{logarithm} appears naturally in the syntomic cohomology group $\text{H}^0_{\text{syn}}(R,\Z_p(1))$ of quasiregular semiperfectoid $\Z_p^{\Cyc}$-algebras $R$.
\end{remark}

	\bibliographystyle{alpha}
	
	{\footnotesize
\bibliography{biblio.bib}
}

\end{document}